\documentclass[11pt,reqno,oneside]{amsart}

\usepackage{graphicx}
\usepackage{amssymb}
\usepackage{epstopdf}

\usepackage[a4paper, total={6in, 9.1in}]{geometry}

\usepackage{amsmath,amsfonts,amsthm,mathrsfs,amssymb,cite}
\usepackage[usenames]{color}

\usepackage{graphics}
\usepackage{framed}

\usepackage{enumerate}
\usepackage{hyperref}
\usepackage{xcolor}
\hypersetup{
  colorlinks,
  linkcolor={blue!80!black},
  urlcolor={blue!80!black},
  citecolor={blue!80!black}
}
\usepackage[capitalize,noabbrev]{cleveref}

\numberwithin{equation}{section}

\newtheorem{Theorem}{Theorem}[section]
\newtheorem{Lemma}[Theorem]{Lemma}
\newtheorem{Proposition}[Theorem]{Proposition}

\newtheorem{Corollary}[Theorem]{Corollary}

\newtheorem{Remark}[Theorem]{Remark}
\numberwithin{equation}{section}

 \def\p{\partial} 
\def \Vh0{\stackrel{\circ}{V}_h}

  \def\f{\frac}  
   \def\eps{\varepsilon}

\def\p{\partial}

\newcommand{\lc}
{\mathrel{\raise2pt\hbox{${\mathop<\limits_{\raise1pt\hbox
{\mbox{$\sim$}}}}$}}}

\newcommand{\gc}
{\mathrel{\raise2pt\hbox{${\mathop>\limits_{\raise1pt\hbox{\mbox{$\sim$}}}}$}}}

\newcommand{\ec}
{\mathrel{\raise2pt\hbox{${\mathop=\limits_{\raise1pt\hbox{\mbox{$\sim$}}}}$}}}

\def\bb{\begin{equation}} \def\ee{\end{equation}}

\def\beqn{\begin{eqnarray}}  \def\eqn{\end{eqnarray}}

\def\beqnx{\begin{eqnarray*}} \def\eqnx{\end{eqnarray*}}

\def\bn{\begin{enumerate}} \def\en{\end{enumerate}}

\def\bd{\begin{description}} \def\ed{\end{description}}

\makeatletter

\newenvironment{figurehere}
  {\def\@captype{figure}}
  {}
\makeatother

\title[Quantum ergodicity and localization of plasmon resonances]{Quantum ergodicity and localization of plasmon resonances}

\author{Habib Ammari}
\address{Department of Mathematics, ETH Z\"urich, R\"amistrasse 101, CH-8092, Switzerland}
\email{habib.ammari@math.ethz.ch}

\author{Yat Tin Chow}
\address{Department of Mathematics, University of California, Riverside, USA}
\email{yattinc@ucr.edu}

\author{Hongyu Liu}
\address{Department of Mathematics, City University of Hong Kong, Hong Kong SAR, China}
\email{hongyu.liuip@gmail.com, hongyliu@cityu.edu.hk}

\begin{document}
\maketitle

\begin{abstract}
We are concerned with the geometric properties of the surface plasmon resonance (SPR).
SPR is a non-radiative electromagnetic surface wave that propagates in a direction parallel to the negative permittivity/dielectric material interface. It is known that the SPR oscillation is very sensitive to the material interface. However, we show that the SPR oscillation asymptotically localizes at places with high magnitude of curvature in a certain sense.  Our work leverages the Heisenberg picture of quantization and 
quantum ergodicity first derived by Shnirelman, Zelditch, Colin de Verdi\`ere and Helffer-Martinez-Robert,
as well as
certain novel and more general ergodic properties of the Neumann-Poincar\'e operator to analyze the SPR field, which are of independent interest to the spectral theory and the potential theory. 
\end{abstract}

\medskip

\noindent{\bf Keywords:}~~surface plasmon resonance, localization, quantum ergodicity, high curvature, Neumann-Poincar\'e operator, quantization  \\ 

\noindent{\bf 2010 Mathematics Subject Classification:}~~58J50, 58J51, 35Q60, 82D80

\section{Introduction}\label{sect:1}

\subsection{Mathematical setup}\label{sect:1.1}
In this work, we are mainly concerned with the plasmonic eigenvalue problem as follows. Let $D$ be an open connected and bounded domain in $\mathbb{R}^d$, $d\geq 2$, with a $\mathcal{C}^{2,\alpha}$, $0<\alpha<1$, boundary $\partial D$ and a connected complement $\mathbb{R}^d\backslash\overline{D}$. Let $\gamma_c$ and $\gamma_m$ be two real constants with $\gamma_m\in\mathbb{R}_+$ given and fixed. Let 
\begin{equation}\label{eq:mc1}
\gamma_{D}  = \gamma_c \chi(D) +  \gamma_m \chi(\mathbb{R}^d \backslash \overline{D}),
\end{equation}
where and also in what follows, $\chi$ stands for the characteristic function of a domain. Consider the following homogeneous problem for a potential field $u\in H_{loc}^1(\mathbb{R}^d)$,
 \beqn
        \nabla \cdot (\gamma_{D} \nabla u) = 0 \ \mbox{in}\; \mathbb{R}^d;\ u(x) =  \mathcal{O}(|x|^{1-d})\ \ \mbox{when}\; d\geq 2 \ \ \mbox{ as }\; |x| \rightarrow \infty,
    \label{transmission}
\eqn
where the last asymptotic holds uniformly in the angular variable $\hat x:=x/|x|\in\mathbb{S}^{d-1}$ and is known as the decay condition. Note that \eqref{transmission} is equivalent to the following transmission problem:
\beqn
    \begin{cases}
        \Delta u = 0 &\hspace*{-2mm}\text{ in }\; D \cup (\mathbb{R}^d\backslash \overline{D} ) \, ,\\[1.5mm]
        u^+ = u^- &\hspace*{-2mm}\text{ on }\; \partial D \, ,\\[1.5mm]
        \gamma_c \f{\p u^{+}}{\p \nu}  =  \gamma_m \f{\p u^{-}}{\p \nu} &\hspace*{-2mm}\text{ on }\; \partial D \, ,\\[1.5mm]
        \mbox{$u$ satisfies the decay} &\hspace*{-2.3mm}\mbox{condition as $|x|\rightarrow\infty$},
    \end{cases}
    \label{transmission2}
\eqn
where $\pm$ signify the traces taken from the inside and outside of $D$ respectively. If there exists a nontrivial solution $u$ to \eqref{transmission2}, then $\gamma_c$ is referred to as a plasmonic eigenvalue and $u$ is the associated plasmonic resonant field. It is apparent that a plasmonic eigenvalue must be negative, since otherwise by the ellipticity of the partial differential operator (PDO) $\mathcal{L}_{\gamma_D} u:=\nabla \cdot (\gamma_D\nabla u)$, \eqref{transmission2} admits only a trivial solution. The plasmonic eigenvalue problem is delicately connected to the spectral theory of the Neumann-Poincar\'e (NP) operator as follows. Let $\Gamma$ be the fundamental solution of the Laplacian in $\mathbb{R}^d$ :
\beqn
    \Gamma (x-y) =
    \begin{cases}
     -\f{1}{2\pi} \log |x-y| & \text{ if }\; d = 2 \, ,\\
     \f{1}{(2-d)\varpi_d} |x-y|^{2-d} & \text{ if }\; d > 2 \, ,
    \end{cases}
    \label{fundamental}
\eqn
with $\varpi_d$ denoting the surface area of the unit sphere in $\mathbb{R}^d$. The Neumann-Poincar\'e (NP) operator $\mathcal{K}^*_{\partial D }: L^2(\partial D, d\sigma) \rightarrow L^2(\partial D , d \sigma)$ is defined by
\beqn
    \mathcal{K}^*_{\partial D} [\phi] (x) := \f{1}{\varpi_d} \int_{\partial D} \f{\langle  x-y,\nu (x) \rangle  }{|x-y|^d} \phi(y) d \sigma(y) \,,
    \label{operatorK}
\eqn
where $\nu(x)$ signifies the unit outward normal at $x \in \partial D$. It is remarked that the NP operator is a classical weakly-singular boundary integral operator in potential theory \cite{book,kellog}. Then a plasmonic resonant field to \eqref{transmission2} can be represented as a single-layer potential:
\beqn
    u(x)=\mathcal{S}_{\partial D} [\phi] (x) := \int_{\partial D} \Gamma(x-y) \phi(y) d \sigma(y),\; x\in\mathbb{R}^d, \label{eq:sl1} 
\eqn
where the density distribution $\phi\in H^{-1/2}(\partial D, d \sigma)$ satisfies 
\begin{equation}\label{eq:eigen1}
\mathcal{K}_{\partial D}^*[\phi]=\lambda(\gamma_c, \gamma_m)\phi,\quad \lambda(\gamma_c, \gamma_m):=\f{\gamma_c+\gamma_m}{2(\gamma_c-\gamma_m)}. 
\end{equation}
That is, in order to determine the plasmonic eigenvalue $\gamma_c$ of \eqref{transmission2}, it is sufficient to determine the eigenvalues of the NP operator $\mathcal{K}_{\partial D}^*$.  On the other hand, in order to understand the peculiar behavior of the plasmonic resonant field, one needs to study the quantitative properties of the NP eigenfunctions in \eqref{eq:eigen1} as well as the associated single-layer potentials in \eqref{eq:sl1}. In this paper, we are mainly concerned with the geometric properties of the plasmonic eigenmode $u$, namely its quantitative relationships to the geometry of $\partial D$. This leads us to establish more general quantum ergodic properties of the singularly integral operators $\mathcal{K}_{\partial D}^*$ and $\mathcal{S}_{\partial D}$. The plasmonic eigenvalue problem is the fundamental basis to the so-called surface plasmon resonance as shall be described in the following. The quantitative understanding of the plasmonic eigenmodes would yield deep theoretical insights on the surface plasmon resonance as well as produce significant physical and practical implications.

\subsection{Physical relevance and connection to existing studies of our results}\label{sect:1.2}

Surface plasmon resonance (SPR) is the resonant oscillation of conducting electrons at the interface between negative and positive permittivity materials stimulated by incident light. It is a non-radiative electromagnetic surface wave that propagates in a direction parallel to the negative permittivity/dielectric material interface. The SPR forms the fundamental basis for an array of industrial and engineering applications, from highly sensitive biological detectors to invisibility cloaks \cite{BS,FM,Kli,LZ,MN,OI,Sch,SPW,Z} through the constructions of different plasmonic devices.  The plasmonics was listed as one of the top ten emerging technologies of 2018 by the Scientific American, stating that ``light-controlled nanomaterials are revolutionizing sensor technology". Next, we briefly discuss the SPR in electrostatics and in Section~\ref{sect:5} we shall extend all the results in the electrostatic case to the scalar wave propagation problem in the quasi-static regime. 

Consider a medium configuration given in \eqref{eq:mc1}, where $\gamma_c$ and $\gamma_m$ respectively specify the dielectric constants of the domain $D$ and the background space $\mathbb{R}^d\backslash\overline{D}$. Let $u_0$ be a harmonic function $\mathbb{R}^d$ which represents an incident field. 
The electrostatic transmission problem is given for an electric potential field $u\in H_{loc}^1(\mathbb{R}^d)$ as follows,
\beqn
    \begin{cases}
        \nabla \cdot (\gamma_{D} \nabla u) = 0 \ \text{ in }\; \mathbb{R}^d, \\[1.5mm]
         u-u_0\ \ \mbox{ satisfies the decay condition as $|x|\rightarrow\infty$}.
    \end{cases}
    \label{transmissionE}
\eqn
Denote the perturbed potential field $u-u_0$ as a single-layer potential \eqref{eq:sl1} with a density function $\phi$ to be determined by the transmission conditions across $\partial D$. Using the following jump relation across $\partial D$: 
\beqn
    \f{\p}{\p \nu} \left(  \mathcal{S}_{\partial D} [\phi] \right)^{\pm} = (\pm \f{1}{2} Id + \mathcal{K}^*_{ \partial D} )[\phi]\,,
    \label{jump_condition}
\eqn
where $Id$ denotes the identity operator, one can show that
\beqn
    \frac{\partial u_0}{\partial \nu} =  \left( \f{\gamma_c+\gamma_m}{2(\gamma_c-\gamma_m)}Id - \mathcal{K}_{\partial D}^* \right) [\phi] \ \ \mbox{on}\ \ \partial D. 
    \label{potential2}
\eqn
Hence, formally there holds
\begin{equation}\label{eq:formally1}
u=u_0+\mathcal{S}_{\partial D}\circ\bigg(\lambda(\gamma_c,\gamma_m)Id-\mathcal{K}_{\partial D}^* \bigg)^{-1}\left[\frac{\partial u_0}{\partial\nu}\bigg|_{\partial D} \right],
\end{equation}
where $\lambda(\gamma_c,\gamma_m)$ is defined in \eqref{eq:eigen1}. Clearly, if $\lambda(\gamma_c, \gamma_m)$ is an eigenvalue to $\mathcal{K}_{\partial D}^*$, then resonance occurs for the boundary integral equation \eqref{potential2}. Consequently, due to a proper incident field $u_0$, resonance can be induced for the electrostatic system \eqref{transmissionE}. This exactly gives the plasmonic eigenvalue problem \eqref{transmission2} and the NP eigenvalue problem \eqref{eq:eigen1}.

According to our discussion above, the spectrum of the NP operator determines the plasmonic eigenvalues $\gamma_c$ through the relationship given by $\lambda(\gamma_c,\gamma_m)$. That is, the spectra of the NP operator determine the negative dielectric constants which can induce the plasmon resonances. Such a connection has aroused growing interest on studying the spectral properties of the NP operator \cite{ACKLM,Amm1,AJ,AKL,BZ,G,shapiro,LL1,plasmon1,weyl2,weyl1}. Since $\partial D$ is smooth, the NP operator $\mathcal{K}_{\partial D}^*$ is compact and hence its eigenvalues are discrete, infinite and accumulating at zero. Moreover, one has $\lambda(\mathcal{K}_{\partial D}^*)\subset (-1/2, 1/2]$. 
It can be directly verified that if $\gamma_c$ and $\gamma_m$ are both positive, then $|\lambda(\varepsilon_c, \varepsilon_m)|>1/2$. In such a case, the invertibility of the operator $( \lambda(\varepsilon_c, \varepsilon_m)I - \mathcal{K}_{\partial D}^*)$
from $L^2(\partial D, d \sigma )$ onto $L^2(\partial D, d \sigma)$ and from
$L_0^2(\partial D, d \sigma)$ onto $L_0^2(\partial D , d \sigma)$ can be proved (cf. \cite{book, kellog}) via the Fredholm theory. This once again necessitates the negativity of the plasmonic eigenvalues $\gamma_c$. It is easily seen, from the properties of $\mathcal{K}_{\partial D}^*$, that the NP eigenvalues are invariant with respect to rigid motions and scaling.
The spectrum of $ \mathcal{K}_{\partial D}^*$ can be explicitly computed for ellipses and spheres \cite{shapiro,seo}. It is worth pointing out that the convergence to zero of those eigenvalues is exponential for ellipses while it is algebraic for spheres. The exponential convergence is critical for the construction of plasmonic devices that can induce invisibility cloaks \cite{ACKLM2,LLL1}. 
Some other computations of Neumann-Poincar\'e eigenvalues as well as the corresponding plasmonic applications for different shapes can be found in \cite{curv_Liu_3,BZ,KLY}.
More recently in \cite{weyl2, weyl1}, it is derived in three dimensions a quantization rule of the NP eigenvalues, showing that the leading-order asymptotic of the $j$-th ordered NP eigenvalue with $j\gg 1$ can be expressed in terms of two global geometric quantities of $\partial D$, namely the Euler characteristic and the Willmore energy. 

It is clear to see that if plasmon resonance occurs, the peculiar behaviors of the resonant field critically depend on the quantitative properties of the NP eigenfunctions. Indeed, according to \eqref{potential2} and via the spectral resolution, the density distribution $\phi$ can be expressed in terms of the NP eigenmodes, and this in turn gives the resonant modes via the single-layer potentials. The SPR field is the superposition of those resonant modes. Hence, in order to gain a thorough understanding of the resonant field, one should carefully study the quantitative properties of the NP eigenfunctions as well as the associated single-layer potentials. However, to our best knowledge, there is little study in the literature on this aspect. It is known that the SPR propagation is very sensitive to the material interface $\partial D$. That is, the SPR is sensitive to any change of the global geometry of $\partial D$. In fact, such a geometric sensitivity forms the fundamental basis of the aforementioned bio-sensing application of SPRs.  Nevertheless, it is speculated that the SPRs may possess certain invariant/robust property related to the local geometry of $\partial D$. Indeed, it is observed in several numerics for some specific geometries \cite{curv_Liu_3} that the SPR waves reveal certain concentration/localization phenomenon at places where the magnitude of the associated curvature is relatively high.  In another recent work \cite{corner}, the behavior of SPR modes around corners of an object is also studied, and concentration/localization around corners are also observed.

The aim of this paper is to rigorously establish the local geometric invariant property of the SPRs in a very general setup. In fact, we show that the SPR wave localizes in a certain sense at places where the magnitude of the associated extrinsic curvature (namely the second fundamental form) is relatively high. Since the SPR depends on $\partial D$ globally, it is highly nontrivial to extract the local geometric information. Nevertheless, we establish a certain more general property of the SPR waves with the help of quantum ergordicity, with which the localization property is a natural consequence of the dynamics of an associated Hamiltonian. In addition to its theoretical significance, our result may have potential applications in generating SPRs that break the quasi-static limit \cite{LL1} and produce plasmonic cloaks \cite{ACKLM,MN}, which are worth investigation in a future study.  Our study also leverages certain novel ergodic properties of the NP operator, which should be of independent interest to the spectral theory and potential theory. 

\subsection{Discussion of the technical novelty}

Finally, we briefly discuss the mathematical strategies of establishing the quantum ergodicity and localization results. As discussed earlier, we need to analyze the geometric structures of the NP eigenfunctions as well as the associated single-layer potentials; that is, the quantitative behaviors of those distributions that are related to the boundary geometry of the underlying domain. Treating those layer-potential operators as pseudo-differential operators, we consider the Hamiltonian flows of the principle symbols of those operators.  In particular, we obtain, via a generalized Weyl's law, that the asymptotic average of the magnitude of the NP eigenfunctions in a neighborhood of each point is directly proportional to a weighted volume of the characteristic variety at the respective point.   Moreover, following the pioneering works of Shnirelman \cite{erg1,erg11}, Zelditch \cite{Zel,erg3,erg32,erg33,erg34,erg35}, Colin de Verdi\`ere \cite{erg2} and Helffer-Martinez-Robert \cite{erg4}
(See also \cite{erg_a,erg_b,trace,Ego}), by considering the Heisenberg picture and lifting the Hamiltonian flow of a principal symbol to a wave propagator, we generalize a result of quantum ergodicity via an application of the ergodicity decomposition theorem to include more general dynamics that may not be ergodic with respect to the Riemannian measure on $ \{ H = 1 \}$ (which is the Liouville measure when $H$ generates the geodesic flow). From that, we obtain a subsequence (of density one) of eigenfunctions such that their magnitude weakly converges to a weighted average of ergodic measures.  This weighted average at different points relates to the volumes of the characteristic variety at the respective points. 
We also provide an upper and lower bounds of the volume of the characteristic variety as functions only depending on the principal curvatures.  
We therefore can characterize the localization of the plasmon resonance by the associated extrinsic curvature at a specific boundary point.
From our result, we have also associated the understanding of plasmon resonances to the dynamical properties of the Hamiltonian flows. For instance, a Hamiltonian circle action will result in a parametrization of ergodic measures by a compact symplectic manifold of dimension $2d-4$ via a symplectic reduction. Our study opens up a new filed with many possible developments on the quantitative properties of plasmon resonances as well as on the spectral properties of Neumann-Poincar\'e type operators. 

The rest of the paper is organized as follows. In Section 2, we briefly discuss the principal symbols of layer-potential operators. In Section 3, we recall the generalized Weyl's law, and generalize the argument of the quantum ergodicity to obtain a variance-like estimate.
In Section 4, we apply the generalized Weyl's law and our generalization of the quantum ergodicity to obtain a comparison result of the magnitude of NP eigenfunction at different points with extrinsic curvature information at the respective points.
Combining these results gives a description of localization of the plasmon resonance around points of high curvature.
We consider extensions to the plasmon resonance in the Helmholtz transmission problem in the quasi-static regime in Section 5.
In Appendix A, we present further discussion upon geometric descriptions of the related Hamiltonian flows.

\section{Symbols of potential operators}\label{sect:2}

In this section, we present the principle symbols of the Neumann-Poincar\'e operator \eqref{operatorK} and the single-layer potential operator $\mathcal{S}_{\partial D}$ in \eqref{eq:sl1} associated with a shape $D$ sitting inside a general space $\mathbb{R}^d$ for any $d\ge 2$.
The special three-dimensional case was first treated in \cite{weyl1,weyl2}, and the general case was considered in \cite{ACL}.
Since this result is of fundamental importance for our subsequent analysis, we shall briefly restate here for the sake of completeness.

We briefly introduce the geometric description of $D\subset\mathbb{R}^d$. Consider a regular parametrization of the surface $\partial D$ as
\beqn
\mathbb{X}: U \subset \mathbb{R}^{d-1} &\rightarrow& \partial D \subset \mathbb{R}^{d}, \notag \\
u = (u_1, u_2, ..., u_{d-1}) &\mapsto &\mathbb{X}(u) \notag \,.
\eqn
For notational sake, we often write the vector $\mathbb{X}_j := \f{\p \mathbb{X}}{\p u_j}$, $j=1,2,\ldots, d-1$.
For a given $d-1$ vector $\{v_j\}_{j = 1}^{d-1}$, we denote
the $d-1$ cross product $\times_{j=1}^{d-1} v_j = v_1 \times v_2 ... \times v_{d-1} $
as the dual vector of the functional $\det(\, \cdot \,, v_1, v_2, ..., v_{d-1} )$, i.e., $\langle w, \times_{j=1}^{d-1} v_i \rangle = \det(w, v_1, v_2, ..., v_{d-1} ) $ for any $w$, whose existence is guaranteed by the Riesz representation theorem.
Then, from the fact that $\mathbb{X}$ is regular, we know that $\times_{j=1}^{d-1} \mathbb{X}_j$ is
non-zero, and the normal vector $\nu := \times_{j=1}^{d-1} \mathbb{X}_j / |\times_{j=1}^{d-1} \mathbb{X}_j | $
is well-defined.
Next, we introduce the following matrix $\mathcal{A}_{ij} (x), x \in \partial D $ , defined as
\beqn
\mathcal{A}(x) := ( \mathcal{A}_{ij}(x) ) =  \langle \textbf{II}_x (\mathbb{X}_i ,\mathbb{X}_j), \nu_x \rangle \notag \,,
\eqn
where $\textbf{II}$ is the second fundamental form given by
\beqn
\textbf{II} : T(\p D) \times T(\p D) &\rightarrow& T^\perp (\p D), \notag \\
\textbf{II}(v,w)&=& - \langle \bar{\nabla}_{v} \nu, w \rangle \nu =  \langle \nu,  \bar{\nabla}_{v}  w \rangle \nu, \notag
\eqn
with $\bar{\nabla}$ being the standard covariant derivative on the ambient space $\mathbb{R}^{d}$.
Moreover, we write $\mathcal{H}(x), x \in \partial D$,  as the mean curvature satisfying 
$$ \text{tr}_{g(x)} (\mathcal{A}(x)) :=  \sum_{i,j = 1}^{d-1} g^{ij}(x) \mathcal{A}_{ji}(x) := (d-1) \mathcal{H} (x) \, , $$
with $(g^{ij}) = g^{-1}$ and $g=(g_{ij})$ being the induced metric tensor. 
From now on, we shall always assume $\mathcal{A}(x) \neq 0$ for all $x \in \partial D$ in this work.
We are now ready to present the principle symbol of $\mathcal{K}^*_{\partial D}$ (cf. \cite{weyl1,weyl2,ACL}).
\begin{Theorem}
The operator $\mathcal{K}^*_{\partial D}$ is a pseudodifferential operator of order $-1$ on $\partial D$ if $\partial D \in \mathcal{C}^{2,\alpha}$ with its symbol given as follows in the geodesic normal coordinate around each point $x$:
\begin{equation}\label{eq:s1}
\begin{split}
p_{\mathcal{K}^*_{\partial D}}(x,\xi) =& p_{\mathcal{K}^*_{\partial D} , -1}(x,\xi)  +  \mathcal{O}(|\xi|^{-2})\\
   =&  (d-1) \mathcal{H}(x) \,  |\xi|^{-1} -  \langle \mathcal{A}(x)  \, \xi, \, \xi \rangle \, |\xi|^{-3}
+  \mathcal{O}(|\xi|^{-2}) \, ,
\end{split}
\end{equation}
where the asymptotic $\mathcal{O}$ depends on $\| \mathbb{X} \|_{\mathcal{C}^2} $.
Hence $\mathcal{K}^*_{\partial D}$ is a compact operator of Schatten $p$ class $S_p$ for $p > d-1$ with $d > 2$.  
\end{Theorem}

\noindent A remark is that the above result holds also for $\mathcal{K}_{\partial D}$ instead of $\mathcal{K}^*_{\partial D}$ when we only look at the leading-order term. Here, $\mathcal{K}_{\partial D}$ signifies the $L^2(\partial D, d \sigma)$-adjoint of the NP operator $\mathcal{K}^*_{\partial D}$. 
We would also like to remark that if geodesic normal coordinate is not chosen, and for a general coordinate, we have instead
\beqnx
p_{\mathcal{K}^*_{\partial D}}(x,\xi) &=&  (d-1) \mathcal{H}(x) \,  |\xi|_{g(x)}^{-1} -  \langle \mathcal{A}(x)  \, g^{-1}(x) \, \xi, \, g^{-1}(x) \,\xi \rangle \, |\xi|_{g(x)}^{-3}
+  O(|\xi|_{g(x)}^{-2}) \,.
\eqnx

From the fact that the Dirichlet-to-Neumann map $\Lambda_0 : H^{1/2}(\partial D, d \sigma) \rightarrow H^{-1/2}(\partial D, d \sigma)$ of the Laplacian in the domain $D \subset \mathbb{R}^d$ satisfies the following \cite{uhlmann}:
\begin{equation}\label{eq:s11}
\begin{split}
p_{\Lambda_0}(x,\xi) = p_{\Lambda_0,1}(x,\xi) +  \mathcal{O}(1) =  |\xi|_{g(x)} +  \mathcal{O}(1) \,,
\end{split}
\end{equation}
together with the jump relation \eqref{jump_condition}, one can handily compute that:
\begin{equation}\label{eq:s2}
\begin{split}
p_{\mathcal{S}_{\partial D}}(x,\xi) = p_{\mathcal{S}_{\partial D},-1}(x,\xi) +  \mathcal{O}(|\xi|_{g(x)}^{-2}) =  \frac{1}{2} |\xi|_{g(x)}^{-1} +  \mathcal{O}(|\xi|_{g(x)}^{-2}) \,.
\end{split}
\end{equation}
We recall the following well-known Kelley symmetrization identity:
\begin{equation}\label{eq:s3}
\mathcal{S}_{\partial D}  \, \mathcal{K}^*_{\partial D} =   \mathcal{K}_{\partial D} \,  \mathcal{S}_{\partial D},
\end{equation}
which indicates that $\mathcal{K}_{\partial D}^*$ is symmetrizable on $H^{-1/2}(\partial D , d \sigma)$ (cf., e.g., \cite{putinar, ando}), i.e. $\mathcal{K}_{\partial D}^*$ is a self-adjoint operator on $L^2_{S_{\partial D}} (\partial D) := ( \overline{ \mathcal{C}^{\infty} (\partial D)}^{\| \cdot \|_{S_{\partial D}} } , \langle \cdot , \cdot \rangle_{S_{\partial D}} ) $, where, for any $f \in L^2_{S_{\partial D}} (\partial D) $,
\[
\| f \|_{S_{\partial D}}^2 := \langle f, f\rangle_{S_{\partial D}} :=  - \langle S_{\partial D} f , f \rangle_{H^{1/2}(\partial D , d \sigma), H^{-1/2}(\partial D , d \sigma)} \,,
\]
is a well-defined inner product for $d\geq 3$ and with a minor modification for $d=2$ (see \cite{ando, Amm1}). 
We remark that there is an equivalence between the two norms $ \| \cdot \|_{S_{\partial D}} $ and $ \| \cdot \|_{H^{-1/2}(\partial D, d \sigma)} $.

Using the symmetrization identity \eqref{eq:s3} (which gives us self-adjointness of the operator $( \mathcal{S}_{\partial D})^{\frac{1}{2}} \, \mathcal{K}^*_{\partial D} ( \mathcal{S}_{\partial D})^{-\frac{1}{2}} $ under the standard $L^2$ inner product) and by comparing the corresponding symbols, together with the fact that $\mathcal{S}_{\partial D}$ is self-adjoint, we have
{\small
\begin{equation}\label{eq:s4}
\begin{split}
\mathcal{K}^*_{\partial D} =& |D|^{-1} \left \{ (d-1) \mathcal{H}(x) \Delta_{\partial D} - \sum_{i,j,k,l = 1}^{d-1} \frac{1}{ \sqrt{|g(x)|} } \partial_i g^{ij}(x) \sqrt{|g(x)|} \mathcal{A}_{jk}(x) g^{kl}(x) \partial_{l} \right\} |D|^{-2} \text{ mod } \Phi \text{SO}^{-2},\medskip \\
\mathcal{S}_{\partial D} =& \frac{1}{2} |D|^{-1}  \text{ mod } \Phi \text{SO}^{-2}. 
\end{split}
\end{equation}
}In \eqref{eq:s4}, $ \Delta_{\partial D}$ is the surface Laplacian of $\partial D$, and $|D|^{-1} := \mathrm{Op}_{|\xi|_{g(x)}^{-1}} $ with $\mathrm{Op}_{a} = \mathcal{F}^{-1} \circ m_{a} \circ\mathcal{F}$ being the action given by the symbol without any large/small parameter, where $\mathcal{F}$ is the Fourier transform (defined via a partition of unity, and is unique modulus $\Phi \text{SO}^{m-1}$ if $a \in \tilde{\mathcal{S}}^m(T^*(\partial D))$) that belongs to the symbol class of order $m$, and $m_a$ is the action with multiplication by the symbol $a$.  We notice that the operator in the curly bracket in \eqref{eq:s4} is itself symmetric.
We therefore have
\beqnx
\mathcal{K}^*_{h,\partial D} := \frac{1}{h} |D|^{-\frac{1}{2}}   \mathcal{K}^*_{\partial D} |D|^{\frac{1}{2}}
\eqnx
being self adjoint up to $\text{mod } h \Phi \text{SO}^{-2}_h$.
Here and also in what follows, $ \Phi \text{SO}^{-m}_h$ is the pseudo-differential operator with action $\mathrm{Op}_{a,h} := \mathcal{F}^{-1}_h \circ m_{a} \circ\mathcal{F}_h$, i.e., with a small parameter $h$ (again uniquely defined modulus $h \Phi \text{SO}_h^{m-1}$ if $a \in \tilde{\mathcal{S}}^m(T^*(\partial D))$) belonging to the symbol class of order $m$. For clarity, the following notations and definitions are used in our study,
\[
\begin{split}
 \bigcup_{i} U_i = \partial D \,,\quad & F_i : \pi_i^{-1} (U_i) \rightarrow U_i \times \mathbb{R}^{d-1}\,,  \quad  \sum_i \psi_i = 1 \,, \quad \text{supp}(\psi_i) \subset U_i\, ; \\
\widetilde{S}^m (T^*(\partial D) )  :=&  \left \{ a: T^*(\partial D)  \backslash \partial D \times \{0\}  \rightarrow \mathbb{C} \, ; \, a = \sum_{i} \psi_i F_i^* a_i, a_i \in \widetilde{S}^m (U_i \times \mathbb{R}^{d-1} \backslash \{0\} ) \right \}; \\
\widetilde{S}^m (U_i \times \mathbb{R}^{d-1} )  := & \bigg\{ a:  U_i \times ( \mathbb{R}^{d-1} \backslash \{0\} )  \rightarrow \mathbb{C} \, ;\\
&\qquad a \in \mathcal{C}^{\infty} (  U_i \times ( \mathbb{R}^{d-1} \backslash \{0\} )  ) \, , \, | \partial_\xi^\alpha \partial_x^{\beta} a (x,\xi) | \leq C_{\alpha, \beta} ( |\xi| )^{ m - |\alpha| } \bigg\}. 
\end{split}
\]

Finally, we note that $\left(\frac{\lambda_i^2}{h},  |D|^{-\frac{1}{2}}  \phi_i\right)$ is an eigenpair of $ [ \mathcal{K}^*_{h,\partial D}  ]^2$ if and only if 
$\left(\lambda_i^2, \phi_i\right)$ is an eigenpair of $ \mathcal{K}^*_{\partial D} $ (cf. \eqref{eq:eigen1}).
Throughout the rest of the paper, we denote
\begin{equation}\label{eq:eg1}
 ( \lambda_i^2(h),  \phi_i(h) ) := \left(\frac{\lambda_i^2}{h},  |D|^{-\frac{1}{2}}  \phi_i\right).
 \end{equation}

\section{Generaalized Weyl's law and quantum ergodicity over the Neumann-Poincar\'e operator}\label{sect:3}

In this section, we recall the concept of quantum ergodicity following the pioneering works of Shnirelman \cite{erg1,erg11}, Zelditch \cite{Zel,erg3,erg32,erg33,erg34,erg35}, Colin de Verdi\`ere \cite{erg2} and Helffer-Martinez-Robert \cite{erg4}, (see also \cite{erg_a,erg_b,trace,Ego}). Although it is a classical theorem, we would still sketch the proofs to some of the materials for the sake of completeness.  Meanwhile, for our subsequent use, we would slightly generalize it via the ergodicity decomposition theorem.

\subsection{Hamiltonian flows of principle symbols}
We begin by considering the following Hamiltonian
\begin{equation}\label{eq:hf1}
\begin{split}
H: T^*(\partial D)  \rightarrow & \mathbb{R} \\ 
H(x, \xi) =& [ p_{\mathcal{K}^*_{\partial D},-1}(x,\xi)  ]^2 \geq 0 \,.
\end{split}
\end{equation}
Note that $T^*(\partial D) $ is endowed with the standard symplectic form $\omega := \sum_{i=1}^{d=1} dx_i \wedge d \xi_i = d \alpha$, where $\alpha := \sum_{i=1}^{d=1} x_i \, d \xi_i  $  is the canonical 1-form.  
Notice that $H$ is only smooth outside $\partial D \times \{0\} \hookrightarrow T^*(\partial D)$.   We now introduce an assumption for our study:

\vskip 2mm

\noindent \textbf{Assumption (A)}  We assume $\langle \mathcal{A}(x) \, g^{-1}(x) \, \omega \, ,\,  g^{-1}(x) \, \omega \rangle \neq (d-1) \mathcal{H}(x)$ for all $x \in \partial D$ and $\omega \in \{ \xi :  |\xi |_{g(x)}^2 = 1 \} \subset T_x^*(\partial D)$.

\vskip 2mm

\noindent As shall be further explored in Appendix \ref{Appendix_A}, this assumption is related to the regularity of the Hamiltonian flow generated by $H$ on the set $\overline{ \{H = 1\} }$.  In particular, it holds if and only if $\overline{ \{H = 1\} } \bigcap \left(  \partial D  \times \{0\} \right)  = \emptyset$.  When $d =3$, this condition implies the strict convexity of $D$.
In the rest of the paper and up till the Appendix, we shall always assume the validity of Assumption (A).   We speculate that this assumption is not necessary for the conclusions of our theorems to hold, and it might be relaxed via techniques introduced in the study of (exotic) non-smooth homogeneous symbol class, e.g. in \cite{homogeneous}. However, we choose to explore along that direction in a future study. 

It can be verified that Assumption (A) is equivalent to the condition that the Hamiltonian $H \neq 0$ everywhere.  
With such an observation and gazing at \eqref{eq:eigen1}, we immediately infer that $\phi $ in \eqref{eq:eigen1} actually sits in $H^{s} (\partial D, d \sigma)$ for all $s$, and hence by the Sobolev embedding, $\phi \in \mathcal{C}^{\infty} (\partial D)$.

Next, we introduce the following auxiliary function:
\beqnx
\rho: \mathbb{R}_+ := \{ r \in \mathbb{R} : r \geq 0 \} &\rightarrow& \mathbb{R}, \\
\rho(r) &=& 1-\exp(-r) \,,
\eqnx
which shall be used in our subsequent analysis.
We realize that $\rho(r)\geq 0$ and $\rho'(r) > 0$ for all $r\in\mathbb{R}_+$.   Moreover, one can verify that $\rho(1/r^2) \in \mathcal{C}^{\infty} (\mathbb{R})$, with $\partial_r^{k}|_{r = 0} \left[ \rho(1/r^2) \right] = 0$ for all $k \in \mathbb{N}$ and 
\[
| \partial_r^{k} \rho(1/r^2) | \leq C_k (1+ |r|^2)^{\frac{-2 - k}{2}} \,.
\]
Using this function, we define 
\begin{equation}\label{eq:hf1_a}
\begin{split}
\tilde{H}: T^*(\partial D)  \rightarrow & \mathbb{R}, \\ 
\tilde{H}(x, \xi) =& \rho( H(x, \xi)  ) \,.
\end{split}
\end{equation}
It can now be handily verified, under Assumption (A), that we have $\tilde{H} \in \mathcal{C}^{\infty} (T^*(\partial D))$, and in fact,  $\tilde{H} \in S^{-2} (T^*(\partial D))$, where $S^m (T^*(\partial D))$ denotes the smooth symbol class of order $m$ defined as
\beqnx
S^m (T^*(\partial D)) & := & \left \{ a: T^*(\partial D) \rightarrow \mathbb{C} \, ; \, a = \sum_{i} \phi_i F_i^* a_i,\ a_i \in S^m (U_i \times \mathbb{R}^{d-1} ) \right \}, \\
S^m (U_i \times \mathbb{R}^{d-1} ) & := & \bigg\{ a: U_i \times \mathbb{R}^{d-1} \rightarrow \mathbb{C} \, ; a \in \mathcal{C}^{\infty} (  U_i \times  \mathbb{R}^{d-1}   ) \\
& &\qquad | \partial_\xi^\alpha \partial_x^{\beta} a (x,\xi) | \leq C_{\alpha, \beta} (1+ |\xi|^2)^{\frac{m - |\alpha| }{2}} \bigg\} \,.
\eqnx

With the above notations, let us consider the following solution under a Hamiltonian flow:
\begin{equation}\label{eq:hf2}
\begin{cases}
\frac{\partial}{\partial t} a(t) &= -\frac{i}{h} \{ \tilde{H} , a (t) \}, \medskip\\
a(0) & = a_0 (x,\xi) \in S^m (T^*(\partial D)),
\end{cases}
\end{equation}
where
$ \{ \cdot , \cdot \}$ is the Poisson bracket defined by 
\beqnx
 \{ f , g\} := X_f \, g = - \omega(X_f, X_g),
\eqnx
with $X_f$ being the symplectic gradient vector field given by
\beqnx
\iota_{X_f} \, \omega =  d f \,.
\eqnx
It is noticed that, away from $\partial D \times \{0\}$, we have
\[
X_{\tilde H} = \rho'(H) X_{\tilde H}, 
\]
where $\rho'(H) > 0$, whereas $X_{\tilde H} = 0$ on $\partial D \times \{0\}$.
With this notion in hand, one readily sees $\frac{\partial}{\partial t} a = X_{\tilde H } a $, and moreover
$ a (t) = a_0 ( \gamma(t), p (t)) $ with
\beqnx
\begin{cases}
 \frac{\partial}{\partial t}  (\gamma(t), p(t) ) &= X_{\tilde H} (\gamma(t), p(t) ),\medskip  \\
(\gamma(0), p(0) ) & = (x, \xi ) \in T^* (\partial D).
\end{cases}
\eqnx
In what follows, we sometimes emphasize the dependence of $a$ on the initial value $(x,\xi)$ by writing
\beqnx
a_{x,\xi}(t) = a(t) \quad  \text{ with } \quad (\gamma(0), p(0) )  = (x, \xi ) \,.
\eqnx
Next we introduce the Heisenberg's picture and lift the above flow to the operator level via Egorov's theorem.  
Since this is a well-known theorem, we only provide a sketch of the proof for the sake of completeness.
\begin{Proposition} \cite{Hor1,Hor2,Ego}
Under Assumption (A), the following operator evolution equation
\begin{equation}\label{eq:ohf1}
\begin{cases}
\frac{\partial}{\partial t} A_h(t) = \frac{\mathrm{i}}{ h } \left[ \mathrm{Op}_{\tilde H,h} , A_h(t) \right],\medskip\\
A_h(0) = \mathrm{Op}_{a_0,h},
\end{cases}
\end{equation}
defines a unique Fourier integral operator (up to $h^{\infty} \, \Phi \mathrm{SO}_h^{-\infty} $) 
\beqnx
A_h(t) = e^{- \frac{\mathrm{i} t}{h}  \mathrm{Op}_{\tilde H,h}  } \, A_h(0) \, e^{ \frac{\mathrm{i} t}{h}  \mathrm{Op}_{\tilde H,h}  }  + O(h \, \Phi \mathrm{SO}_h^{\,m-1} )  
\eqnx
for $t < C \log (h)$.
Moreover,
\beqnx
A_h(t) =  \mathrm{Op}_{a (t) ,h} + \mathcal{O}(h \, \Phi \mathrm{SO}_h^{\,m-1} )  \,,
\eqnx
or that $ p_{A_h(t)}(x,\xi) = a_{x,\xi}(t)  + \mathcal{O}(|\xi|^{-1})\,.$
\end{Proposition}
\begin{proof} 
The existence of the solution to the equation \eqref{eq:ohf1} comes from first constructing the symbol in the principle level by noting that $[ \mathrm{Op}_{a} , \mathrm{Op}_{b} ] = \mathrm{Op}_{ \{a,b\}} + \mathcal{O}(h  \Phi\mathrm{SO}_h^{m+n-2} )$ if $a \in S^{m} ( T^*(\partial D))$ and $b \in S^{n} ( T^*(\partial D))$. Then one inductively constructs the full symbol, and bounds the error operator via the Calder\'on-Vaillancourt theorem repeatedly. By Beal's theorem, the operator is guaranteed as an FIO. 

The proof of both expressions of $A_h(t)$ comes from checking that the principle symbols coincide, and then using the Zygmund trick to bound the error operator.
\end{proof}

Let us consider $\tilde{H}(x,\xi)$ given as in \eqref{eq:hf1_a}, i.e. $ \tilde{H}(x,\xi)= \rho \left( [p_{\mathcal{K}^*_{\partial D,-1}}(x,\xi)  ]^2 \right) $.  Then we immediately have
\beqnx
 \mathrm{Op}_{H,h}  = \rho \left( [ \mathcal{K}^*_{h,\partial D} ]^2 \right) \text{ mod } (h \Phi\mathrm{SO}_h^{-3} ) \, ,
\eqnx
and hence the following corollary holds.

\begin{Corollary}
\label{congugation} Under Assumption (A), the symbol of
\beqn 
A_h(t) = e^{- \frac{\mathrm{i} t}{h}  \rho \left( [ \mathcal{K}^*_{h,\partial D} ]^2 \right)  }  A_h(0) e^{ \frac{\mathrm{i} t}{h} \rho \left( [ \mathcal{K}^*_{h,\partial D} ]^2 \right)}  + \mathcal{O}(h \Phi \mathrm{SO}_h^{-1} )  \,
\eqn
is given by
\beqn
 p_{A_h(t)}(x,\xi) = a_{x,\xi}(t)  + \mathcal{O}(|\xi|^{-1})\,.
\eqn
\end{Corollary}

\subsection{Trace formula and generalized Weyl's law}    

We first state the Schwartz functional calculus as follows. 
\begin{Lemma} \cite{Hor1,Hor2}
Let $\mathcal{S}(\mathbb{R})$ denote the space of Schwartz functions on $\mathbb{R}$. Then for $f \in \mathcal{S}(\mathbb{R})$, $f ( \mathrm{Op}_{a,h} ) \in \Phi \mathrm{SO}_h^{-\infty}$ and 
\begin{equation}\label{eq:sfc1}
 f ( \mathrm{Op}_{a,h} ) = \mathrm{Op}_{f ( a )} + \mathcal{O}( h \Phi \mathrm{SO}_h^{-\infty} ).
\end{equation}
\end{Lemma}
\begin{proof}
The theorem can be proved via an almost holomorphic extension of $f$ to $f^{\mathbb{C}}$ (e.g. by H\"ormander \cite{Hor1,Hor2}) and the Helffer-Sj\"ostrand formula
$
f(A) = \frac{1}{2 \pi \mathrm{i}} \int_{C} \partial_{\overline{z}} f^{\mathbb{C}} (z - A)^{-1} dz \wedge d \overline{z} \,.
$
\end{proof}

We proceed to state the following trace theorem, and again give only a brief sketch of the proof for the sake of completeness.
\begin{Proposition}\cite{Hor1,Hor2,trace,erg1,Zel,erg3,erg2}
Given $a \in S^m(T^*(\partial D))$, if $\mathrm{Op}_{a,h} $ is in the trace class and $f \in \mathcal{S}(\mathbb{R})$, then
\beqnx
(2 \pi h )^{ (d-1)} \mathrm{tr} ( f(  \mathrm{Op}_{a,h}  ) ) =   \int_{T^*(\partial D)} f( a ) \, d \sigma \otimes d \sigma^{-1} + \mathcal{O}(h) \, ,
\eqnx
where $d \sigma \otimes d \sigma^{-1}$ is the Liouville measure given by the top form $\omega^{d-1} / (d-1) !$.
\end{Proposition}
\begin{proof}
For notational convenience, we first consider the Weyl quantization $\mathrm{Op}^{w}_{a,h}$ instead.
We have from the Schwartz kernel theorem that
\[
 [\mathrm{Op}^{w}_{f(a),h} (\phi) ] (x) = \int_{\partial D} K_h (x,y) \phi(y) d \sigma(y)
\]
for some $ K_h (x,y) \in \mathcal{D}' (\partial D \times \partial D )$, which actually possesses the following explicit expression
\[
(2 \pi h )^{ (d-1)} K_h (x,y) =   \mathcal{F}_h \left[  a\left(\frac{ x+y}{2} , \cdot \right) \right] (x-y) + \mathcal{O}(h) \,,
\]
via the partition of unity and the local trivialisation (by an abuse of notation).  Then from the functional calculus we have
\beqnx
\text{tr} (  f (\mathrm{Op}^{w}_{a,h} )  ) = \mathrm{tr} (   \mathrm{Op}^{w}_{f(a),h}   )  = \int_{\partial D} K_h(x,x) d \sigma \,.
\eqnx
To conclude our theorem, we notice that the Weyl quantization $\text{Op}^{w}_{a,h}$ and left/right quantizations $\text{Op}^{L/R}_{a,h}$ differ only in the higher order term after an application of the operator $\exp\left(\pm i \frac{h}{2} \partial_x \partial_\xi\right)$, and our choice of quantization here is $\text{Op}_{a,h} := \text{Op}^{R}_{a,h}$.
\end{proof}
\noindent 
Combining functional calculus and trace theorem gives the following generalized Weyl's law, together with the fact that $ ( \lambda_i^2(h),  \phi_i(h) )$ is an eigenpair of $ [ \mathcal{K}^*_{h,\partial D}  ]^2$ if and only if $ ( \rho\left( \lambda_i^2(h) \right),  \phi_i(h) )$ is an eigenpair of $ \rho \left( [ \mathcal{K}^*_{h,\partial D}  ]^2 \right)$.

\begin{Proposition} \cite{Hor1,Hor2,trace,erg1,Zel,erg3,erg2}
Under Assumption (A), fixing $r \leq s$, for any $a \in S^{m}(T^*(\partial D))$, we have as $h\rightarrow+0$,
\beqn
(2 \pi h )^{ (d-1)}  \sum_{r \leq \lambda_i^2 (h) \leq s} c_i  \, \langle  \mathrm{Op}_{a,h}  \, \phi_i(h) , \phi_i(h) \rangle_{L^2(\partial D, d \sigma)}  =  \int_{ \{ r \leq H \leq  s\} } a  \, d \sigma \otimes d \sigma^{-1} + o_{r,s}(1),
\label{weyl}
\eqn
where $c_i := | \phi_i |_{H^{-\frac{1}{2}}(\partial \Omega, d \sigma)}^{-2} $ is the $H^{-\frac{1}{2}}$ semi-norm and the little-$o$ depends on $r,s$.
\end{Proposition}
\begin{proof}
Take $ f_{\varepsilon} \left( \rho\left( [ \mathcal{K}^*_{h,\partial D} ]^2 \right) \right)$ where $f_{\varepsilon} \in \mathcal(S)$ approximating $\chi_{[\rho(r),\rho(s)]}$. Then $f_{\varepsilon} \left( \rho\left( [ \mathcal{K}^*_{h,\partial D} ]^2 \right) \right) \in \Phi \text{SO}_h^{-\infty}$ by the functional calculus with the trace formula 
\begin{equation}\label{eq:wyl1}
\begin{split}
& (2 \pi h )^{(d-1)}  \text{tr} \left( _{\varepsilon} \left( \rho\left( [ \mathcal{K}^*_{h,\partial D} ]^2 \right) \right) \, \text{Op}_{a,h}   \, _{\varepsilon} \left( \rho\left( [ \mathcal{K}^*_{h,\partial D} ]^2 \right) \right) \right)\\
 =&  \int_{T^*(\partial D)} a f_{\varepsilon}^2( \rho( H ) )  \, d \sigma \otimes d \sigma^{-1} + \mathcal{O}_{r,s,\varepsilon}(h) \,,
\end{split}
\end{equation}
where $\mathcal{O}$ depends on $r,s,\varepsilon$.
Passing $\varepsilon$ to $0$ in \eqref{eq:wyl1}, $ f_{\varepsilon} \left( \rho\left( [ \mathcal{K}^*_{h,\partial D} ]^2 \right) \right)$ converges to the spectral projection operator, and $f_\varepsilon^2 (\rho(H) )$ converges to $\chi_{ \{ \rho(r) \leq \rho(H) \leq  \rho(s) \} } = \chi_{ \{ r \leq H \leq  s\} }$, which readily yields \eqref{weyl}.  Finally, we notice that $
 \| \phi_i (h) \|_{L^{2}(\partial D , d \sigma)} =  \| \phi_i \|_{H^{-\frac{1}{2}}(\partial D, d \sigma)} $.
\end{proof}

We would like to point out that if taking $a =1$ in \eqref{weyl}, it leads us back to the well-known Weyl's law:
\begin{Corollary} \cite{Hor1,Hor2,trace,erg1,Zel,erg3,erg2} Under Assumption (A), we have
\beqn
 \sum_{r \leq \lambda_i^2 (h) \leq s}  1  =  (2 \pi h )^{ 1-d }  \int_{ \{ r \leq H \leq  s\} } \, d \sigma \otimes d \sigma^{-1} + o_{r,s}(h^{1-d}) \,.
\label{weyl2}
\eqn
\end{Corollary}

\subsection{Ergodic decomposition theorem and quantum ergodicity} 

Let us denote $\sigma_H $ as the Riemannian measure on $ \{ H = 1 \} \subset T^* (\partial D)$.
Since $X_{H} H = 0 $ and $\mathcal{L}_{X_{H}} \omega^{d-1} = 0 $, we have that
 $\sigma_{H} := \lim_{\varepsilon \rightarrow 0} \varepsilon^{ d-2} \, \chi_{\{ | H - 1 | < \varepsilon \} }   \, d \sigma \otimes d \sigma^{-1} $ is an invariant measure on  $ \{ H = 1 \} $.
We also notice that, on $ \{ H = 1 \} $, 
\[
X_{\tilde H} = \rho'(1) X_{H} = e^{-1} X_{H} \,.
\]
Let $M_{X_{H}} (  \{ H =1 \}  )$ be the set of $X_{H}$ invariant measures on  $\{ H = 1 \} $ and also $M_{X_{H},\text{erg}} (  \{ H = 1 \}  )$ be the set of ergodic measures with respect to the Hamiltonian flow generated by $X_H$ on $\{ H = 1 \} $.
Since $X_{\tilde H} = e^{-1} X_{H} $, we have $M_{X_{\tilde H}} (  \{ H =1 \}  ) = M_{X_{H}} (  \{ H =1 \}  )$ and  $M_{X_{\tilde H},\text{erg}} (  \{ H = 1 \}  ) = M_{X_{H},\text{erg}} (  \{ H = 1 \}  )$.  Therefore, we do not distinguish between them.
Now, since $\{ H = 1 \}$ has a countable base, the weak-* topology of $M_{X_{H}} (  \{ H = 1 \}  )$ is metrizable, and hence Choquet's theorem can be applied to obtain the following ergodic decomposition theorem.
\begin{Proposition} \cite{ergodic} \label{full_decomposition}
Given a probability measure $\eta \in M_{X_{H}} (  \{ H = 1 \}  )$, there exists a probability measure $\nu \in M ( M_{X_{H},\text{erg}} (  \{ H = 1 \}  ) )$ such that
\[
\eta = \int_{M_{X_{H},\text{erg}} (  \{ H = 1 \}  ) }  \mu \,  d \nu (\mu) \,.
\]
\end{Proposition}

 Applying Proposition~\ref{full_decomposition} to $\sigma_{H} / \sigma_H (\{ H = 1 \}) $, we have a probability measure $\nu \in M ( M_{X_{H},\text{erg}} (  \{ H = 1 \}  ) )$ such that
\[
\sigma_{H} = \sigma_H (\{ H = 1 \}) \int_{M_{X_{H},\text{erg}} (  \{ H = 1 \}  ) }  \mu \,  d \nu (\mu) \,.
\]
Note by rescaling $\{H = E \} = E^{-1/2} \{H = 1\}$, and therefore $d \sigma \otimes d \sigma^{-1} = E^{1 - \frac{d}{2} } d E \otimes d \sigma_{H}$. For any $\mu \in  M_{X_{H},\text{erg}} (  \{ H = 1 \}  )  $, we let $\mu_E : = m_{ E^{-1/2} } \# \mu  \in M_{X_{H},\text{erg}} (  \{ H = E \}  )$ be the push-forward measure given by $m_{ E^{-1/2} } : T^*(\partial D) \rightarrow   T^*(\partial D) \, (x,\xi) \mapsto (x, E^{-1/2}  \xi) $, then
\begin{eqnarray}
 \sigma \otimes \sigma^{-1}   = \sigma_H (\{ H = 1 \}) \int_{ (0,\infty] \times M_{X_{H},\text{erg}} (  \{ H = 1 \}  ) }   \mu_E  \, E^{1 - \frac{d}{2} } \, ( d E  \otimes  d \nu ) \, ( E, \mu) \,.
\label{disintegration}
\end{eqnarray}
\begin{Remark}
\eqref{disintegration} is a disintegration of the measure $ \sigma \otimes \sigma^{-1} $ into a family of measures $ \{ \mu_E \} $ parametrized by $(E,\mu)$. The $ E^{1 - \frac{d}{2} } dE \otimes  d \nu$-measurable measure-valued function $(E,\mu) \mapsto \mu_E$ integrates to produce the original measure $ \sigma \otimes \sigma^{-1} $.
\end{Remark}
\noindent 
Next, we aim to derive a more general version of quantum ergodicity, following the original argument in, e.g. \cite{erg1,erg11,Zel,erg3,erg32,erg33,erg34,erg35,erg2,erg_a,erg_b,trace} as follows. 
To start with, we have the following application from Birkhoff-Khinchin \cite{birkhoff} and Von-Neumann's ergodic theorems \cite{vonneumann}.
\begin{Lemma} 
 \label{birkroff} 
Under Assumption (A), for any $r\leq s$ and all $a_0 \in \mathcal{S}^m(T^*(\partial D))$,
we have
\beqnx
\frac{1}{T}\int_0^T a_{x,\xi}(t) dt \rightarrow_{a.e. d \sigma \otimes d \sigma^{-1} \text{ and } L^2(\{ r \leq H \leq s\}, d \sigma \otimes d \sigma^{-1} ) } \bar{a}(x,\xi)\text{ as $T \rightarrow \infty$,}
\eqnx
for some $X_H$-invariant $\bar{a} \in L^2 (\{ r \leq H \leq s\}, d \sigma \otimes d \sigma^{-1} ) $, where a.e.-$ E^{1 - \frac{d}{2} } dE \otimes  d \nu$, we have
\beqnx
\bar{a} (x,\xi ) = \int_{\{H = E\} } a_0   \, d \mu_E  \quad   \text{ a.e. } d \mu_E .
\eqnx
\end{Lemma}

\begin{Remark}
This lemma should be understood in the sense of disintegration \eqref{disintegration} of $ \sigma \otimes \sigma^{-1} $ via $(E,\mu) \mapsto \mu_E$, as well as under a choice of representative of an equivalence class in $L^2 (\{ r \leq H \leq s\}, d \sigma \otimes d \sigma^{-1} )$: there exists an $X_H$-invariant function $\bar{a}(x,\xi )$ as a representative in an $L^2$-equivalence class, such that, for a.e.-$ E^{1 - \frac{d}{2} } dE \otimes  d \nu$ over the parameter $(E,\mu)$, $\bar{a}(x,\xi )$ is a constant a.e. $\mu_E$ over $(x,\xi)$, and the constant is given by $ \int_{\{H = E\} } a_0   \, d \mu_E$.
\end{Remark}

\begin{proof}
By Birkhoff and Von-Neumann's ergodic theorems \cite{birkhoff,vonneumann} on $\chi_{ \{ r \leq H \leq s\}} \,  d \sigma \otimes d \sigma^{-1}  $, we have
\beqnx
\frac{1}{T}\int_0^T a_{x,\xi}(t) dt \rightarrow_{a.e. d \sigma \otimes d \sigma^{-1} \text{ and } L^2(\{ r \leq H \leq s\}, d \sigma \otimes d \sigma^{-1} ) } \bar{a}(x,\xi)\text{ as $T \rightarrow \infty$,}
\eqnx
for some  $\bar{a} \in L^2 (\{ r \leq H \leq s\}, d \sigma \otimes d \sigma^{-1} ) $ invariant under the Hamiltonian flow. 
Let 
\[ \mathcal{E} := \left \{ (x,\xi) \in \{ r \leq H \leq s \}  : \limsup_T \left| \frac{1}{T}\int_0^{T} a_{x,\xi}(t) dt - \bar{a}(x,\xi) \right | > 0   \right \} \,, \] then $ \sigma \otimes \sigma^{-1} ( \mathcal{E} ) = 0$.  Now by Lemma \ref{full_decomposition}, we have
\beqnx
 \sigma_H (\{ H = 1 \}) \int_{ [r,s] \times M_{X_{H},\text{erg}} (  \{ H = 1 \}  ) }   \mu_E( \mathcal{E} )  \, E^{1 - \frac{d}{2} } \, d E  \otimes  d \nu \, ( E, \mu) =  \sigma \otimes \sigma^{-1} ( \mathcal{E} ) = 0 \,,
\eqnx
and therefore, a.e. $ E^{1 - \frac{d}{2} } dE \otimes  d \nu$, we have $\mu_E (\mathcal{E} ) = 0 $.  Meanwhile, since $X_H$ is ergodic with respect to $\mu_E$, by Birkhoff's theorem \cite{birkhoff}, we obtain
\beqnx
\frac{1}{T}\int_0^T a_{x,\xi}(t) dt \rightarrow_{a.e. \mu_E} 
 \int_{\{H = E\} } a_0   \, d \mu_E  \quad  \text{ as $T \rightarrow \infty$.}
\eqnx
Again let 
\[ \mathcal{E}_{\mu_E} := \left \{ (x,\xi) \in \{ r \leq H \leq s \}  : \limsup_T \left | \frac{1}{T}\int_0^{T} a_{x,\xi}(t) dt -   \int_{ \{ H = E \} } a_0 \, d \mu_E \right | > 0   \right \} \,, \]
we have $\mu_E ( \mathcal{E}_{\mu_E}) = 0 $.  Therefore, a.e. $ E^{1 - \frac{d}{2} } dE \otimes  d \nu$, $\mu_E \left(\mathcal{E} \bigcup \mathcal{E}_{\mu_E} \right) = 0 $.  The lemma follows by the uniqueness of the limit.
\end{proof}

We can then show the following theorem by following the arguments in \cite{erg1,erg11,Zel,erg3,erg32,erg33,erg34,erg35,erg2,erg_a,erg_b,trace,Ego} with some generalizations.
\begin{Theorem}\label{thm:ergo1}
Under Assumption (A), fixing  $r \leq s$ and writing $c_i := | \phi_i |_{H^{-\frac{1}{2}}(\partial D, d \sigma)}^{-2} $ to denote the $H^{-\frac{1}{2}}$ semi-norm, we have the following (variance-like) estimate as $h\rightarrow+0$,
\begin{equation}\label{eq:est1}
\frac{1}{ \sum_{ r \leq  \lambda_i^2(h) \leq s} 1 }
 \sum_{ r \leq  \lambda_i^2(h) \leq s}  c_i^2 \left|   \langle  A_h \,  \phi_i(h) ,  \phi_i(h) \rangle_{L^2(\partial D, d \sigma)}  -  \langle \mathrm{Op}_{\bar{a}, h}   \,  \phi_i(h) ,  \phi_i(h) \rangle_{L^2(\partial D , d \sigma)}   \right|^2 \rightarrow 0.
 \end{equation}
\end{Theorem}

\begin{proof}
We lift the Birkhoff and Von-Neumann to the operator level, via the Hamiltonian flow of the principle symbol.  Consider $A_h(0) = A_h$. From the definition of $\phi_i(h)$, we have for each $i$
\begin{equation}\label{eq:est2}
\begin{split}
& \langle A_h(t) \phi_i(h) , \phi_i(h) \rangle_{L^2(\partial D, d \sigma)} \\
=& \langle A_h(0) e^{- \frac{\mathrm{i} t}{h} \rho \left( [ \mathcal{K}^*_{h,\partial D} ]^2 \right) }  \phi_i(h) , e^{- \frac{\mathrm{i} t}{h}  \rho \left( [ \mathcal{K}^*_{h,\partial D} ]^2 \right) }  \phi_i(h) \rangle_{L^2(\partial D, d \sigma)}   + \mathcal{O}_t (h) \\
=&  \langle A_h \,  \phi_i(h) ,  \phi_i(h) \rangle_{L^2(\partial D, d \sigma)}+ \mathcal{O}_t (h) ,
\end{split}
\end{equation}
where the second equality comes from Corollary \ref{congugation} and the definition of the NP eigenfunctions (cf. \eqref{eq:eg1}), and the asymptotic $\mathcal{O}$ depends on $t$.  Averaging both sides of \eqref{eq:est2} with respect to $T$, we have
\beqnx
 \left \langle   \left( \frac{1}{T}\int_0^T  A_h(t) dt \right) \phi_i(h) , \phi_i(h) \right \rangle_{L^2(\partial D, d \sigma)}  =  \langle A_h   \phi_i(h) ,  \phi_i(h) \rangle_{L^2(\partial D , d \sigma)} + \mathcal{O}_T (h),
\eqnx
where the asymptotic $\mathcal{O}$ depends on $T$.  Then, again by Corollary \ref{congugation}, it is handy to verify that
\beqnx
 \frac{1}{T}\int_0^T  A_h(t) dt -   \text{Op}_{\bar{a}, h} = \mathrm{Op}_{ \frac{1}{T}\int_0^T a(t) dt  - \bar{a} } + \mathcal{O}_T (h )
\eqnx
is a pseudo-differential operator.  Next from the Cauchy-Schwarz inequality, we have
\begin{equation}\label{eq:sum1}
\begin{split}
 &\left|  \frac{ \left \langle   \text{Op}_{\bar{a}, h}  \phi_i(h) , \phi_i(h) \right \rangle_{L^2(\partial D, d \sigma)} }{{\langle \phi_i(h) , \phi_i(h) \rangle_{L^2(\partial D, d \sigma)}} } - \frac{ \langle  A_h    \phi_i(h) ,  \phi_i(h) \rangle_{L^2(\partial D, d \sigma)} }{{\langle \phi_i(h) , \phi_i(h) \rangle_{L^2(\partial D, d \sigma)}} } \right|^2 \\
\leq & \frac{  \left \langle  \left(  \frac{1}{T}\int_0^T  A_h(t) dt -   \text{Op}_{\bar{a}, h} \right)^* \left( \frac{1}{T}\int_0^T  A_h(t) dt -   \text{Op}_{\bar{a}, h} \right)  \phi_i(h) ,  \phi_i(h) \right \rangle_{L^2(\partial D, d \sigma)} }{{\langle \phi_i(h) , \phi_i(h) \rangle_{L^2(\partial D, d \sigma)}} } + \mathcal{O}_T (h^2).   
\end{split}
\end{equation}
Therefore, summing up $i$ of \eqref{eq:sum1} and applying \eqref{weyl} and \eqref{weyl2}, we have
\begin{equation}\label{eq:ee4}
\begin{split}
& \frac{1}{ \sum_{ r \leq  \lambda_i^2(h) \leq s} 1 }
 \sum_{ r \leq  \lambda_i^2(h) \leq s}   c_i^2 \left|   \langle  A_h \,  \phi_i(h) ,  \phi_i(h) \rangle_{L^2(\partial D , d \sigma)}  -  \langle \text{Op}_{\bar{a}, h}   \,  \phi_i(h) ,  \phi_i(h) \rangle_{L^2(\partial D , d \sigma)}   \right|^2 \\
\leq&
\frac{\int_{ \{ r \leq H \leq  s\} } \left| \frac{1}{T}\int_0^T a_{(x,\xi)}(t) dt  - \bar{a} \right|^2 \, d \sigma \otimes d \sigma^{-1}}{\int_{ \{ r \leq H \leq  s\} } \, d \sigma \otimes d \sigma^{-1}} + o_{r,s,T}(1).
\end{split}
\end{equation}
Finally, \eqref{eq:est1} readily follows by noting that the first term at the right-hand side of \eqref{eq:ee4} goes to zero as $T$ goes to infinity.

The proof is complete. 
\end{proof}

With Theorem~\ref{thm:ergo1}, together with Chebeychev's trick and a diagonal argument, we have the following quantum ergodicity result \cite{erg1,erg11,Zel,erg3,erg32,erg33,erg34,erg35,erg2,erg_a,erg_b,trace} with some generalization.
\begin{Corollary}
Under Assumption (A), given $r,s$, there exists $S(h) \subset J(h) := \{i \in \mathbb{N} : r  \leq \lambda_i^2(h) \leq s   \}$ such that 
for all $a_0 \in S^{m}(T^*(\partial D))$, we have as $h\rightarrow +0$,
\beqn
\max_{i \in S(h)} c_i \left|   \left \langle ( A_h - \mathrm{Op}_{\bar{a}, h} )  \,  \phi_i(h) ,  \phi_i(h) \right \rangle_{L^2(\partial D , d \sigma)} \right| = o_{r,s}(1) \, \text{ and } \, \frac{\sum_{i \in S(h) } 1 }{ \sum_{i \in J(h) } 1 } = 1 + o_{r,s}(1) \,.
\label{ergodicity}
\eqn
\end{Corollary}
\noindent A very important remark of the above corollary is that the set $S(h) $ is chosen independently of the choice of $a_0$.

\section{Locolization/concentration of plasmon resonances in electrostatics}

In this section, we are ready to present one of our main results on the localization and concentration of plasmon resonances in electrostatics. 

\subsection{Consequences of generalized Weyl's law and quantum ergodicity} 

 We first derive the following theorem to characterise the local behavior of the NP eigenfunctions and their relative magnitude. In what follows, we let $(\lambda_i, \phi_i)$, $i=1, 2, \ldots$, be the ordered eigenpairs to \eqref{eq:eigen1}. We denote $\sigma_{x,H}$ as the Riemannian measure on $ \{ H(x,\cdot) = 1 \} \subset T^*_x (\partial D)$. By the generalized Weyl's law in Section~\ref{sect:3}, we can first show the following key result in our study. 

\begin{Theorem}
\label{theorem1}
Given any $x \in \partial D$, we consider $ \{ \chi_{x,\delta} \}_{\delta > 0}$ being a family of smooth nonnegative bump functions compactly supported in $B_{\delta} (x)$ with $ \int_{\partial D} \chi_{p,\delta} \, d \sigma = 1$.
Under Assumption (A), fixing $r \leq s$, $\alpha\in\mathbb{R}$ and $p, q \in \partial D$, there exists a choice of $\delta(h)$ depending on $r,s,p,q$ and $\alpha$ such that, as $h\rightarrow +0$, we have $\delta(h) \rightarrow 0$ and
\beqn
\frac{ \sum_{r \leq \lambda_i^2 (h) \leq s} c_i  \int_{\partial D} \chi_{p,\delta (h) }(x) | | D |^{\alpha}   \, \phi_i (x) |^2 d \sigma(x) }{ \sum_{r \leq \lambda_i^2 (h) \leq s} c_i  \int_{\partial D} \chi_{q,\delta (h) }(x) | | D |^{\alpha}   \, \phi_i (x) |^2 d \sigma(x) } =  \frac{     \int_{ \{ H (p, \cdot) = 1 \} }  | \xi |_{g(p)}^{1+ 2 \alpha } d \sigma_{p,H}  }{    \int_{ \{ H (q, \cdot) = 1 \} }  | \xi |_{g(q)}^{1+ 2 \alpha } d \sigma_{q,H} } + o_{r,s,p,q,\alpha}(1),
\label{concentration1}
\eqn
where $c_i := | \phi_i |_{H^{-\frac{1}{2}}(\partial D , d \sigma)}^{-2} $ is the $H^{-\frac{1}{2}}$ semi-norm and the little-$o$ depends on $r,s,p,q$ and $\alpha$.  In particular, if $\alpha = - \frac{1}{2}$, the right-hand side of \eqref{concentration1} is the ratio between the volumes of the two varieties at the respective points.
\end{Theorem}
\begin{proof}
Taking $p\in \partial D$, we consider $a(x,\xi) := \chi_{p,\delta}(x)  | \xi |^{1+ 2 \alpha}_{g(x)}  $ in \eqref{weyl}.
With this, together with the fact that $\mathrm{Op}_{a,h} = h^{1+ 2 \alpha}  | D |^{1/2+ \alpha} \mathrm{Op}_{\chi_{p,\delta}(x),h} | D |^{1/2+ \alpha} - h \mathrm{Op}_{\tilde{a}_{p,\delta},h} $ for some $\tilde{a}_{p,\delta} \in S^{2 \alpha}(T^*(\partial D)) $, we have
\begin{equation}\label{eq:nnn1}
\begin{split}
&  (2 \pi h )^{ (d  + 2 \alpha ) }  \sum_{r \leq \lambda_i^2 (h) \leq s} c_i  \int_{\partial D} \chi_{p,\delta}(x) | | D |^{\alpha}   \, \phi_i (x) |^2 d \sigma(x)  \\
= & \int_{ \{ r \leq H \leq  s\} } \chi_{p,\delta}(x)  | \xi |^{1+ 2 \alpha}_{g(x)}  \, d \sigma \otimes d \sigma^{-1} +  h \int_{ \{ r \leq H \leq  s\} } \tilde{a}_{p,\delta} \, d \sigma \otimes d \sigma^{-1} + o_{r,s,\alpha}(1) \\
\end{split}
\end{equation}
after applying \eqref{weyl} once more upon $\tilde{a}_{p,\delta}$.
With \eqref{eq:nnn1}, we have, after choosing another point $q \in \partial D$ and taking a quotient between the two, that
\beqnx
& & \frac{ \sum_{r \leq \lambda_i^2 (h) \leq s} c_i  \int_{\partial D} \chi_{p,\delta}(x) | | D |^{\alpha}   \, \phi_i (x) |^2 d \sigma (x) }{ \sum_{r \leq \lambda_i^2 (h) \leq s} c_i  \int_{\partial D} \chi_{q,\delta}(x) | | D |^{\alpha}   \, \phi_i (x) |^2 d \sigma (x) } \\
&=&  \frac{ \int_{ \{ r \leq H \leq  s\} } \chi_{p,\delta}(x)  | \xi |^{1+ 2 \alpha}_{g(x)}  \, d \sigma \otimes d \sigma^{-1} +  h \int_{ \{ r \leq H \leq  s\} } \tilde{a}_{p,\delta} \, d \sigma \otimes d \sigma^{-1} }{ \int_{ \{ r \leq H \leq  s\} } \chi_{q,\delta}(x)  | \xi |^{1+ 2 \alpha}_{g(x)}  \, d \sigma \otimes d \sigma^{-1} +  h \int_{ \{ r \leq H \leq  s\} } \tilde{a}_{q,\delta}\, d \sigma \otimes d \sigma^{-1}} + o_{r,s,\alpha}(1).  
\eqnx
Now, for any given $h$, we can make a choice of $\delta (h)$ depending on $r,s,p,q,\alpha$ such that as $h \rightarrow +0$, we have $\delta (h) \rightarrow 0$ (much slower than $h$) and 
\[
\left| h \int_{ \{ r \leq H \leq  s\} } \tilde{a}_{p,\delta (h)} \, d \sigma \otimes d \sigma^{-1}  \right| 
+  \left| h \int_{ \{ r \leq H \leq  s\} } \tilde{a}_{q,\delta (h)} \, d \sigma \otimes d \sigma^{-1} \right|  \rightarrow 0 \,.
\]
We also realize as $h \rightarrow +0$, with this choice of $\delta (h)$ that $\delta (h) \rightarrow 0$. One in fact has for $y = p,q$ that
\[
 \int_{ \{ r \leq H \leq  s\} } \chi_{y,\delta(h) }(x)  | \xi |^{1+ 2 \alpha}_{g(x)}  \, d \sigma \otimes d \sigma^{-1}  \rightarrow  \int_{ \{ r \leq H (y, \cdot) \leq  s\} } | \xi |^{1+ 2 \alpha}_{g(y)}  d \sigma^{-1} \,.
\]
Therefore, we have
\beqnx
\frac{ \sum_{r \leq \lambda_i^2 (h) \leq s} c_i  \int_{\partial D} \chi_{p,\delta(h) }(x) | | D |^{\alpha}   \, \phi_i (x) |^2 d \sigma(x) }{ \sum_{r \leq \lambda_i^2 (h) \leq s} c_i  \int_{\partial D} \chi_{q,\delta (h) }(x) | | D |^{\alpha}   \, \phi_i (x) |^2 d \sigma(x) } =  \frac{   \int_{ \{ r \leq H (p, \cdot) \leq  s\} } | \xi |^{1+ 2 \alpha}_{g(p)}  d \sigma^{-1}  }{   \int_{ \{ r \leq H (q, \cdot) \leq  s\} } | \xi |^{1+ 2 \alpha}_{g(q)}  d \sigma^{-1}  } + o_{r,s,p,q,\alpha}(1)  \,.
\eqnx
To conclude our theorem, we see that for all $y = p,q$,
\[
 \int_{ \{ r \leq H (y, \cdot) \leq  s\} } | \xi |^{1+ 2 \alpha}_{g(y)}  d \sigma^{-1} = \left(  \int_{r}^s E^{-\frac{1}{2} - \alpha - \frac{d}{2} } d E \right) \left(  \int_{ \{ H (y, \cdot) = 1 \} }  | \xi |_{g(y)}^{1+ 2 \alpha } d \sigma_{y,H} \right) \,.
\]

The proof is complete. 
\end{proof}

Theorem~\ref{theorem1} states that, given $p,q \in \partial D$, the relative magnitude between a $c_i$-weighted sum of a weighed average of $| |D|^{\alpha} \phi_i |^2$ over a small neighborhood of $p$ to that of $q$ asymptotically depends on the ratio between the weighted volume of $ \{ H (p, \cdot) =1 \} $ and that of $\{ H (q, \cdot) =1 \} $. This is critical for our subsequent analysis since it reduces our study to analyzing the aforementioned weighted volumes. 

\begin{Theorem}
\label{theorem2}
Under Assumption (A), there is a family of distributions $\{ \Phi_{\mu}  \}_{M_{X_{H},\text{erg}}} \in \mathcal{D}' ( \partial D \times \partial D)$ as the Schwartz kernels of $\mathcal{K} _{\mu}$ such that they form the following partition of the identity operator $Id$:
\begin{equation}\label{eq:pu1}
Id = \int_{M_{X_{H},\text{erg}}(\{ H = 1\})}  \mathcal{K}_{\mu} \,  d \nu \, (\mu ) \, ,
\end{equation}
in the weak operator topology satisfying that for any given $r, s$, there exists $S(h) \subset J(h) := \{i \in \mathbb{N} : r h \leq \lambda^2_i  \leq s h  \}$ such that for all $\varphi \in \mathcal{C}^\infty ( \partial D)$ and as $h\rightarrow +0$,
\begin{equation}\label{eq:pu2}
\max_{i \in S(h)}  \left|   \int_{ \partial D}  \varphi(x) \left( c_i \, | |D|^{-\frac{1}{2}} \phi_i (x) |^2 - \int_{M_{X_{H},\text{erg}}(\{ H = 1\})}  \mu (x)  g_i (\mu) d \nu(\mu)  \right) d \sigma(x)
 \right| = o_{r,s}(1)  \,.
\end{equation}
In \eqref{eq:pu1} and \eqref{eq:pu2},
\begin{equation}\label{eq:pu3}
\begin{split}
&\qquad g_{i} (\mu) :=  c_i \left \langle \mathcal{K}_{\mu}  |D|^{-\frac{1}{2}} \phi_i , |D|^{-\frac{1}{2}} \phi_i \right \rangle_{L^2(\partial D, d \sigma)} \,, \\
&  \int_{M_{X_{H},\text{erg}}(\{ H = 1\})}  g_i (\mu) d \nu(\mu) = 1 \, , \, \
\frac{\sum_{i \in S(h) } 1 }{ \sum_{i \in J(h) } 1 } = 1 + o_{r,s}(1) \,,
\end{split}
\end{equation}
and moreover,
\begin{equation}\label{eq:pu4}
\begin{split}
 &\qquad\quad \mu (p)  \geq 0 \,, \, \int_{\partial D} \mu (p)  d \mu(p) = 1\,, \, \\
  & \frac{ \int_{M_{X_{H},\text{erg}}(\{ H = 1\})}   \mu (p)  d \nu(\mu) }{ \int_{M_{X_{H},\text{erg}}(\{ H = 1\})}   \mu (q)  d \nu(\mu) }  =  \frac{ \int_{\{ H(p,\cdot)=1\}} d \sigma_{p,H} }{ \int_{\{ H(q,\cdot)=1\}} d \sigma_{q,H} } \text{ a.e. } (d \sigma \otimes d \sigma) (p,q) \,.
\end{split}
\end{equation}
\end{Theorem}

\begin{proof}
Let $f, \varphi \in \mathcal{C}^{\infty}(\partial D)$ be given.  Let us consider $a(x,\xi) := \varphi(x) $. Then we have 
\[
 \int_{\{ H=E\} } \varphi d \mu_E  =  \int_{\{ H=1\} } \varphi d \mu .
 \]  
Take a partition of unity $\{\chi_i\}$ on $\{U_i\}$. With an abuse of notation via identification of points with the local trivialisation $\{F_i\}$, we have by Lemmas \ref{full_decomposition} and \ref{birkroff} that
\begin{equation}\label{eq:pp1}
\begin{split}
& [ \sigma_H (\{ H = 1 \}) ]^{-1} [ \text{Op}_{\bar{\varphi} , h} f ](y) \\
=& \sum_l [ \sigma_H (\{ H = 1 \}) ]^{-1} [ \text{Op}_{\bar{\varphi} \chi_l , h} f ](y) \\
=& \int_{ (0,\infty] \times M_{X_{H},\text{erg}} (  \{ H = 1 \}  ) } \sum_l \left(  \int_{\{H=E\}} \exp(x-y, \xi / h )  \bar{a}(x,\xi)  \chi_l (x) f (x)  d \mu_E \right) \\
&\hspace*{4cm}\times  E^{1 - \frac{d}{2} } \, (d E  \otimes  d \nu ) \, ( E, \mu ) \\
=&  \int_{ (0,\infty] \times M_{X_{H},\text{erg}} (  \{ H = 1 \}  ) } \sum_l \left( \int_{\{ H=1\} } \varphi d \mu \right ) \left( \int_{\{H=E\}} \exp(x-y, \xi / h )   \chi_l (x) f (x)  d \mu_E \right) \\
&\hspace*{4cm} \times  E^{1 - \frac{d}{2} } \, ( d E  \otimes  d \nu ) \, ( E, \mu ) .
\end{split}
\end{equation}
On the other hand, considering $Id = \mathrm{Op}_{1,h} = \text{Op}_{ \bar{1}, h}$ (which is independent of $h$), we can show that
\beqnx
&  & [ \sigma_H (\{ H = 1 \}) ]^{-1}   [ \text{Op}_{ 1, h} f ](y)  \\
&=&  \int \limits_{ (0,\infty] \times M_{X_{H},\text{erg}} (  \{ H = 1 \}  ) }  \sum_l  \left( \int_{\{H=E\}} \exp(x-y, \xi / h )   \chi_l (x) f (x)  d \mu_E \right)       \, E^{1 - \frac{d}{2} }   \, ( d E  \otimes  d \nu ) \, ( E, \mu )  \,.
\eqnx
If we define $ \mathcal{K}_{\mu}$ (which is again independent of $h$) to be such that
\beqnx
[ \sigma_H (\{ H = 1 \}) ]^{-1}   [ \mathcal{K}_{\mu} f ](y)  =  \int_{ (0,\infty]  }  \sum_l \left( \int_{\{H=E\}} \exp(x-y, \xi / h )   \chi_l (x) f (x)  d \mu_E \right)       \, E^{1 - \frac{d}{2} }   \, d E\, ( E )  \, ,
\eqnx
then we have by definition that
\[
Id = \int_{M_{X_{H},\text{erg}}(\{ H = 1\})}  \mathcal{K}_{\mu} \,  d \nu \, (\mu ) \, 
\]
in the weak operator topology. That is, 
\[
 \langle f , f  \rangle_{L^2(\partial D , d \sigma)} =  \int_{M_{X_{H},\text{erg}}(\{ H = 1\})}  \langle  \mathcal{K}_{\mu}  f , f  \rangle_{L^2(\partial D , d \sigma)}  \,  d \nu \, (\mu ) \,,
\]
and 
\beqnx
 \langle \text{Op}_{  \overline{ \varphi }  , h} f , f  \rangle_{L^2(\partial D , d \sigma)}  =   \int_{M_{X_{H},\text{erg}}(\{ H = 1\})}  
 \int_{\{ H=1\} }  \varphi \, \langle  \mathcal{K}_{\mu} f , f \rangle_{L^2(\partial D , d \sigma)}   \, d \mu \,  d \nu \, (\mu ) \,.
\eqnx
Recall that $d \sigma_{H}(x,\xi) / \sigma_H (\{ H = 1 \}) =  d \mu(x,\xi) \,  d \nu (\mu) $ is a probability measure. We now apply the disintegration theorem to the measure $d \mu (x,\xi) \, d \nu (\mu) $ and obtain a disintegration $d \mu_p (x,\xi) \, d\nu(\mu) \otimes d \sigma(p) $, where the measure-valued map $ ( \mu , p ) \mapsto \mu_p$ is a $d\nu \otimes d \sigma $ measurable function together with $\mu_p \left( \{ H = 1 \} \backslash (  \{ H(p,\cdot) = 1 \} \bigcap \text{spt}(\mu) ) \right) = 0 $ a.e. $d\nu \otimes d \sigma$.  Therefore, we obtain
\beqnx
 \langle \text{Op}_{  \overline{ \varphi }  , h} f , f  \rangle_{L^2(\partial D , d \sigma)} &= &  \int_{\partial D}  \int_{M_{X_{H},\text{erg}}(\{ H = 1\})}  
 \int_{\{ H=1\} }  \varphi \, \langle  \mathcal{K}_{\mu} f , f \rangle_{L^2(\partial D , d \sigma)}   \, d \mu_p \,  (d\nu \otimes d \sigma) \, (\mu ,p ) \,.
\eqnx
We next observe that
\[
\int_{\{ H=1\} }  \varphi \, d \mu_p  =  \int_{\{ H(p,\cdot)=1\} }  \varphi \, d \mu_p = \varphi(p) \, \mu_p( \{ H=1\} ) .
\]
If we denote
\[
\mu(p) := \mu_p( \{ H=1\} ) \geq 0 \,,
\]
then a.e. $d \nu (\mu)$, the function $\mu(p) \in L^1(\partial \Omega, d \sigma)$. As a result of the disintegration, we have a.e. $d \nu (\mu)$,
\[
\int_{\partial \Omega} \mu(p) d \sigma(p) = \mu(\{H = 1\}) = 1 \,.
\]
Furthermore, we have
\beqnx
 \langle \text{Op}_{  \overline{ \varphi }  , h} f , f  \rangle_{L^2(\partial D , d \sigma)} &= &  \int_{\partial D}  \int_{M_{X_{H},\text{erg}}(\{ H = 1\})}  
 \varphi(x) \, \mu(x) \, \langle  \mathcal{K}_{\mu} f , f \rangle_{L^2(\partial D , d \sigma)} \,  (d\nu \otimes d \sigma) \, (\mu ,x ) \,.
\eqnx

Now, we choose $f = \phi_i (h) = |D|^{-\frac{1}{2}} \phi_i$ and apply \eqref{ergodicity} to obtain the conclusion of our theorem. It is noted that the choice of $S(h)$ is independent of $\varphi \in \mathcal{C}^{\infty} (\partial D)$.
The ratio in the last line of the theorem comes from the fact that a.e. $ d \sigma(p)$ we have by definition
\[
 \int_{M_{X_{H},\text{erg}}(\{ H = 1\})}   \mu (p)  d \nu(\mu) =  \int_{M_{X_{H},\text{erg}}(\{ H = 1\})}   \mu_p( \{ H(p,\cdot)=1\} ) d \nu(\mu) =  \frac{ \int_{\{ H(p,\cdot)=1\}} d \sigma_{p,H}  }{ \sigma_H(\{H =1\})} \,.
\]

The proof is complete. 
\end{proof}

Theorem~\ref{theorem2} indicates that most of the function $ c_i ||D|^{-\frac{1}{2}} \phi_i |^2$ weakly converges to a $g_{i} (\mu) \, d \nu(\mu)$-weighted average of  $ \mu(p) $,
where the ratio between a $d \nu(\mu)$-weighted average of $ \mu(p) $ and that of $ \mu(q) $ 
depends on the ratio between the volume of $ \{ H (p, \cdot) =1 \} $ and that of $\{ H (q, \cdot) =1 \} $.

For the sake of completeness, we also give the original version of the quantum ergodicity:
\begin{Corollary}
\label{corollary3}
Under Assumption (A), if the Hamiltonian flow given by $X_H$ is furthermore ergodic with respect to $\sigma_{H}$ on $ \{ H = 1 \}$, then given $r,s$ , there exists $S(h) \subset J(h) := \{i \in \mathbb{N} : r h \leq \lambda^2_i  \leq s h  \}$, such that for all $\varphi \in \mathcal{C}^{\infty} (\partial D)$ and as $h\rightarrow +0$,
\begin{equation}
\begin{split}
& \max_{i \in S(h)} \left|   \int_{ \partial D }  \varphi(x) \left( c_i | |D|^{-\frac{1}{2}} \phi_i (x) |^2 - \frac{ \int_{\{ H(x,\cdot)=1\}} d \sigma_{x,H}  }{ \sigma_H(\{H =1\})}  \right) d \sigma(x)
 \right| = o_{r,s}(1),\\
&\hspace*{3cm}\frac{\sum_{i \in S(h) } 1 }{ \sum_{i \in J(h) } 1 } = 1 + o_{r,s}(1) \,.
\end{split}
\end{equation}
\end{Corollary}
\begin{proof}
The conclusion follows by noting that if $X_H$ is ergodic with respect to $\sigma_H$, then $\sigma_H \in M_{X_{H},\text{erg}} (  \{ H = 1 \}  )$ and we can take $ \nu = \delta_{\sigma_H} $ which is the Dirac measure of $\sigma_H \in M_{X_{H},\text{erg}} (  \{ H = 1 \}  )$. In this case, we obtain
\[
\begin{split}
& \int_{M_{X_{H},\text{erg}}(\{ H = 1\})}   \mu (x)  d \nu(\mu)  =  \int_{M_{X_{H},\text{erg}}(\{ H = 1\})}   \mu (x)  d \delta_{\sigma_H} (\mu)\\
 =&  \sigma_H (x) = \int_{\{ H(x,\cdot)=1\}} d \sigma_{x,H}  \big/ \int_{\{ H=1\}} d \sigma_H .
 \end{split}
\]
\end{proof}
By Corollary~\ref{corollary3}, we see that if $X_H$ is ergodic with respect to $\sigma_H$, most of the function $ c_i ||D|^{-\frac{1}{2}} \phi_i |^2$ weakly converges to the volume of the characteristic variety $ \{ H (x, \cdot) =1 \} $ up to a constant.
It is interesting to remark that we expect that the above argument can be extended to the comparison between $ c_i ||D|^{\alpha} \phi_i |^2$, and we choose to investigate along that direction in a future study.

\subsection{Localization/concentration of plasmon resonance at high-curvature points}

From Theorems~\ref{theorem1} and \ref{theorem2} in the previous subsection, it is clear that the relative magnitude of the NP eigenfunction $\phi_i $ at a point $x$  depends on the (weighted) volume of the characteristic variety $\{H(x,\cdot) =1\}$.   Therefore, in order to understand the localization of plasmon resonance, it is essential to obtain a better description of this volume.  It turns out that this volume heavily depends on the magnitude of the second fundamental forms $A(x)$ at the point $x$.   As we will see in this subsection, in general, the higher the magnitude of the second fundamental forms $A(x)$ is, the larger the volume of the characteristic variety becomes. In particular, in a relatively simple case when the second fundamental forms at two points are constant multiple of each other, we have the following volume comparison.
\begin{Lemma}
\label{theorem3}
Let $p,q \in \partial D$ be such that $A(p) = \beta A(q) $ for some $\beta > 0$ and $g(p) = g(q)$. Then $|\{ H(p,\cdot) = 1 \} | = \beta^{d-2} |\{ H(q,\cdot) = 1\} | $.    We also have $$ \int_{ \{ H (p, \cdot) =1 \} }  | \xi |_{g(p)}^{1+ 2 \alpha }  d \sigma_{p,H} = \beta^{d- 1+ 2 \alpha } \int_{ \{ H (q, \cdot) =1 \} }  | \xi |_{g(q)}^{1+ 2 \alpha }   d \sigma_{q,H} . $$
\end{Lemma}
\begin{proof}
From $-2$ homogeneity of $H$, we have $H(p,\xi  ) = H(q,\xi / \beta )$, and therefore 
$\{ H(p,\xi)  = 1 \}  = \beta \{ H(q,\xi)\} = 1 \} $, which readily yields the conclusion of the theorem. 
\end{proof}

A better understanding of the localization can be achieved by a more delicate volume comparison of the characteristic variety at different points with the help of Theorems \ref{theorem1} and \ref{theorem2} and Corollary \ref{corollary3}.  However, it is less easy to give a more explicit comparison of the volumes between $\{H(p,\cdot) =1\}$ and $\{H(q,\cdot) =1\}$ by their respective second fundamental forms $A(p)$ and $A(q)$. The following lemma provides a detour to control how the (weighted) volume of $\{H(p,\cdot) =1\}$ depends on the principal curvatures $\{ \kappa_i(p) \}_{i=1}^{d-2}$.
\begin{Lemma}
\label{theorem_volume}
Let $F : \mathbb{R}^{d-2} \rightarrow \mathbb{R}$ be given as
\begin{equation}\label{eq:d1}
F_{\alpha} \left( \{ \kappa_i \}_{i=1}^{1-2} \right)  := \int_{\mathbb{S}^{d-2}}   \left| \sum_{i=1}^{d-1} \widetilde{\kappa_i } \omega_i^2 \right|^{d-1 + 2 \alpha} \sqrt{ \sum_{i=1}^{d-1} \widetilde{\kappa_i }^2 \omega_i^2 }\ d \omega,
\end{equation}
where 
\begin{equation}\label{eq:d2}
 \widetilde{\kappa_i } := \sum_{j=1}^{d-1} \kappa_j - \kappa_i. 
\end{equation}
Then we have the following inequality:
\begin{equation}\label{eq:d3}
F_{\alpha} \left( \{ \kappa_i(p) \}_{i=1}^{1-2} \right)  \leq \int_{ \{ H(p,\cdot) = 1 \} } | \xi |_{g(p)}^{1+2\alpha} \, d \sigma_{p,H} \leq 2 F_{\alpha} \left( \{ \kappa_i(p) \}_{i=1}^{d-2} \right).
\end{equation}
\end{Lemma}
\begin{proof}
We first simplify the expression of $H(p,\xi) = 0$ by fixing a point $p$ and choosing a geodesic normal coordinate with the principal curvatures along the directions $\xi_i$.  In this case
$$H(p,\xi) = \left(   \sum_{i=1}^{d-1} \widetilde{\kappa_i (p) } \, \xi_i^2 \right)^2 \bigg/  \left(   \sum_{i=1}^{d-1} \xi_i^2 \right)^3 \,. $$
Let us parametrize the surface $\{ H(p,\cdot) = 1 \}$ by $\omega \in \mathbb{S}^{d-2}$ with $\xi (\omega) := r(\omega) \, \omega$, which is legitimate due to the $-2$ homogeneity of $H$ with respect to $\xi$.  With this, we readily see that on $H = 1$ one has 
$$
r(\omega) =  \sum_{i=1}^{d-1} \widetilde{\kappa_i (p) } \, \omega_i^2  \,.
$$
Hence by virtue of the Sherman-Morrison formula one has
\begin{equation}\label{eq:d4}
\begin{split}
L_{ij} :=& \frac{ \partial \xi }{\partial \omega}_{ij} = r(\omega) \delta_{ij} + 2 \widetilde{\kappa_{j}(p)} \omega_j \omega_i  \, ,\\
 ( L^{-T} )_{ij} =& \frac{1}{r(\omega)} \delta_{ij} -\frac{2}{3 r(\omega)} \widetilde{\kappa_{i}(p)} \omega_i \omega_j  \,,\, \det(L) = 3 \left(r(\omega)\right)^{d} \,,
\end{split}
\end{equation}
with which, via a change of variables, one can further derive that
\[
| \xi |_{g(p)}^{1+2\alpha}  d \sigma_{p,H} = | r (\omega)|^{d -1+2\alpha } \sqrt{4 \left( \sum_{i=1}^{d-1} \widetilde{\kappa_i (p) }^2 \, \omega_i^2 \right)^2 - 3 \left( \sum_{i=1}^{d-1} \widetilde{\kappa_i (p) } \, \omega_i^2 \right) \left( \sum_{i=1}^{d-1} \omega_i^2 \right)  }  d \omega \,.
\]
Finally by the Cauchy-Schwarz inequality, we therefore have
\[
\begin{split}
 & \left| \sum_{i=1}^{d-1} \widetilde{\kappa_i(p) } \omega_i^2 \right|^{d-1 + 2 \alpha} \sqrt{ \sum_{i=1}^{d-1} \widetilde{\kappa_i (p)}^2 \omega_i^2 } d \omega\\
  \leq & | \xi |_{g(p)}^{1+2\alpha}  d \sigma_{p,H}  \leq 2    \left| \sum_{i=1}^{d-1} \widetilde{\kappa_i (p)} \omega_i^2 \right|^{d-1 + 2 \alpha} \sqrt{ \sum_{i=1}^{d-1} \widetilde{\kappa_i  (p)}^2 \omega_i^2 }  d \omega  \,,
  \end{split}
\]
which readily completes the proof.
\end{proof}

Lemma~\ref{theorem_volume} supplies us with a strong tool to obtain the comparison between the ratio of the magnitude of the eigenfunctions via the magnitudes of the principal curvatures at the respective points.  For instance, if it happens that $\min_i |\widetilde{\kappa_i (p)}| \gg \max_i |\widetilde{\kappa_i (q)}| $, then it is clear that the weighted volume of  $\{H(p,\cdot) =1\}$ is much bigger than that at $q$.

\vskip 3mm

\begin{Remark}
As we will explore in Appendix \ref{Appendix_A}, when $d = 3$, $\overline{ \{H = 1\} } \bigcap \left(  \partial D  \times \{0\} \right)  = \emptyset$ if and only if $A(p) > c_0 I$ for all $x \in \partial D$.    In this strictly convex case with $d=3$, we therefore have 
\begin{equation}\label{eq:d5}
\min_{i=1,2} \kappa_i^{3+2\alpha} (p)  \leq F_{\alpha} \left( \{ \kappa_i(p) \}_{i=1}^{1-2} \right) \leq \max_{i=1,2} \kappa_i^{3+2\alpha} (p),
\end{equation}
and hence
\begin{equation}\label{eq:d6}
\min_{i=1,2} \kappa_i^{3+2\alpha} (p)  \leq \int_{ \{ H(p,\cdot) = 1 \} } | \xi |_{g(p)}^{1+2\alpha} \, d \sigma_{p,H} \leq 2 \max_{i=1,2} \kappa_i^{3+2\alpha} (p) \,.
\end{equation}
This fully captures the geometric behavior that the NP eigenfunctions localize at the point in a high curvature when $d = 3$ and when the flow on $ \{H = 1\} $ is non-singular.
Further remarks and brief discussions upon certain geometric properties of the Hamiltonian flow is postponed to Appendix \ref{Appendix_A}.
\end{Remark}

Finally, we discuss the implication of the localization/concentration result of the NP eigenfunctions to the surface plasmon resonances. According to our discussion in Section~\ref{sect:1.2}, an SPR field $u$ is the superposition of the plasmon resonant modes of the form 
\begin{equation}\label{eq:sss1}
u=\sum_{i}\alpha_i \mathcal{S}_{\partial D}[\phi_i],
\end{equation}
 where $\alpha_i\in\mathbb{C}$ represents a Fourier coefficient and each $\phi_i$ is an NP eigenfunction, namely $\mathcal{K}_{\partial D}[\phi_i]=\lambda_i\phi_i$ with $\lambda_i\in\mathbb{R}$ being an NP eigenvalue. As it is known in the literature, a main feature of the SPR field is that it exhibits a highly oscillatory behavior (due to the resonance) and the resonant oscillation is mainly confined in a vicinity of the boundary $\partial D$. For a boundary point $p\in\partial D$, one handily computes from \eqref{jump_condition} that
 \begin{equation}\label{eq:sss2}
 \frac{\partial}{\partial\nu}\left(\mathcal{S}_{\partial D}[\phi_i] \right)^{\pm}(p)=(\pm \frac 1 2 I +\mathcal{K}^*_{\partial D})[\phi_i](p)=(\pm\frac 1 2+\lambda_i)\phi_i(p). 
 \end{equation}
 Generically, \eqref{eq:sss2} indicates that if $|\phi_i(p)|$ is large, then $|\nabla \mathcal{S}_{\partial D}[\phi_i]|$ is also large in a neighbourhood of $p$.  
 Hence, by the localization/concentration of the NP eigenfunction $\phi_i$ established above for a high-curvature point $p\in\partial D$, it is unobjectionable to see from \eqref{eq:sss2} that the resonant energy of the plasmon resonant mode $\mathcal{S}_{\partial D}[\phi_i]$ also localizes/concentrates near the point $p$, in the sense that the resonant oscillation near $p$ is more significant than that near the other boundary point with a relatively smaller magnitude of curvature. According to our earlier analysis following Theorem~\ref{theorem1}, this is particularly the case for the high-mode-number plasmon resonant mode, namely $\mathcal{S}_{\partial D}[\phi_i]$ with $i\in\mathbb{N}$ sufficiently large, which corresponds to that $\lambda_i$ is close to the accumulating point $0$. Consequently, one can readily conclude similar localization/concentration results for the SPR field $u$ in \eqref{eq:sss1}.   
 
%
%
%

\section{Localization/concentration of plasmon resonances for quasi-static wave scattering}\label{sect:5}

In this section, we consider the scalar wave scattering governed by the Helmholtz system in the quasi-static regime and extend all of the electrostatic results to this quasi-static case. Let $\varepsilon_0, \mu_0, \varepsilon_1 , \mu_1$ be real constants and in particular, assume that $\varepsilon_0$ and $\mu_0$ are positive. Let $D$ be given as that in Section~\ref{sect:1}, and set
\[
\mu_{D}  = \mu_1 \chi(D) +  \mu_0 \chi(\mathbb{R}^d \backslash \overline{D}),\quad 
\varepsilon_{D}  = \varepsilon_1 \chi(D) +  \varepsilon_0 \chi(\mathbb{R}^d \backslash \overline{D}). 
\]
$(\varepsilon_1, \mu_1)$ and $(\varepsilon_0, \mu_0)$, respectively, signify the dielectric parameters of the plasmonic particle $D$ and the background space $\mathbb{R}^d\backslash\overline{D}$. Let $\omega\in\mathbb{R}_+$ denote a frequency of the wave. We further set $k_0  := \omega \sqrt{ \varepsilon_0 \mu_0} $ and $k_1  := \omega \sqrt{ \varepsilon_1 \mu_1} $, where we would take the branch of the square root with non-negative imaginary part (in the case that $\varepsilon_1\mu_1$ is negative). Let $u_0$ be an entire solution to $(\Delta+k_0^2) u_0=0$ in $\mathbb{R}^d$. Consider the following Helmholtz scattering problem for $u\in H_{loc}^1(\mathbb{R}^d)$ satisfying
\beqn
    \begin{cases}
        \nabla \cdot (\frac{1}{\mu_{D} } \nabla u) + \omega^2 \varepsilon_{D} u = 0 &\text{ in }\; \mathbb{R}^d, \\[1.5mm]
         (\frac{\partial}{\partial |x| } - \mathrm{i} k_0 ) (u - u_0) = o(|x|^{ - \frac{d-1}{2}}) &\text{ as }\; |x| \rightarrow \infty,
    \end{cases}
    \label{transmissionk}
\eqn
where the last limit is known as the Sommerfeld radiation condition that characterises the outgoing nature of the scattered field $u-u_0$. The Helmholtz system \eqref{transmissionk} can be used to describe the transverse electromagnetic scattering in two dimensions, and the acoustic wave scattering in three dimensions. Nevertheless, we unify the study for any dimension $d\geq 2$. Moreover, we are mainly concerned with the quasi-static case, namely $\omega \ll 1$, or equivalently $k_0\ll 1$.

Similar to the electrostatic case, we next introduce the integral formulation of \eqref{transmissionk}. To that end, we first introduce the associated layer potential operators as follows. Let 
\beqn
    \Gamma_k (x-y) := C_{d} ( k|x-y| )^{- \frac{d-2}{2}} H^{(1)}_{\frac{d-2}{2}}(k|x-y|) ,
    \label{fundamental2}
\eqn
be the outgoing fundamental solution to the differential operator $\Delta+k^2$, where $C_d$ is some dimensional constant and  $H^{(1)}_{\frac{d-2}{2}}$ is the Hankel function of the first kind and order $(d-2)/2$.

We introduce the following single and double-layer potentials associated with a given wavenumber $k\in\mathbb{R}_+$, 
\beqn
    \mathcal{S}^k_{\partial D} [\phi] (x) &:=&  \int_{\partial D} \Gamma_k(x-y) \phi(y) d \sigma(y) ,\ \ x\in\mathbb{R}^d,\\
  \mathcal{D}^k_{\partial D} [\phi] (x) &:=& \int_{\partial D} \frac{\partial }{\partial \nu_y } \Gamma_k(x-y) \phi(y) d \sigma(y) ,\ \ x\in\mathbb{R}^d \setminus \partial D.
\eqn
The single-layer potential $\mathcal{S}^k_{\partial D}$ satisfies the following jump relation on $\partial D$ (cf. \cite{book, kellog}):
\beqn
    \f{\p}{\p \nu} \left(  \mathcal{S}^k_{\partial D} [\phi] \right)^{\pm} = (\pm \f{1}{2} Id + {\mathcal{K}^k_{\partial D}}^* )[\phi]\,,
    \label{jump_condition2}
\eqn
where the superscripts $\pm$ indicate the traces from outside and inside of $D$, respectively, and
${\mathcal{K}^k_{\partial D}}^*: L^2(\partial D) \rightarrow L^2(\partial D)$ is the Neumann-Poincar\'e (NP) operator of wavenumber $k$ defined by
\beqn
    {\mathcal{K}^k_{\partial D}}^* [\phi] (x) := \int_{\partial D} \partial_{\nu_x} \Gamma_k(x-y) \phi(y) d \sigma(y) \, .
    \label{operatorK2}
\eqn
With this, $u\in H_{loc}^1(\mathbb{R}^d)$ in \eqref{transmissionk} can be given by
\begin{equation}\label{eq:ss}
u = 
\begin{cases}
u_0 + \mathcal{S}^{k_0}_{\partial D} [\psi] & \text{ on } \mathbb{R}^d \backslash \overline{D}, \\
\mathcal{S}^{k_1}_{\partial D} [\phi]  & \text{ on }D,
\end{cases}
\end{equation}
where $(\phi,\psi) \in L^2 (\partial D) \times L^2 (\partial D)$ is formally given by (provided that $k_1^2$ is not a Dirichlet eigenvalue of the Laplacian in $D$)
\beqnx
\begin{cases}
\mathcal{S}^{k_1}_{\partial D} [\phi]  - \mathcal{S}^{k_0}_{\partial D} [\psi]  = u_0,\medskip \\
\frac{1}{\mu_1}( - \f{1}{2} Id + {\mathcal{K}^{k_1}_{\partial D}}^* )[\phi] - \frac{1}{\mu_0}( \f{1}{2} Id + {\mathcal{K}^{k0}_{\partial D}}^* ) [\psi] =  \frac{1}{\mu_0}   \f{\p u_0}{\p \nu}   \,,
\end{cases}
\eqnx
or that
\begin{equation}\label{eq:ss1}
\begin{split}
 &\left\{   \f{1}{2}  \left( \frac{1}{\mu_0} Id + \frac{1}{\mu_1} \, \left( \mathcal{S}^{k_1}_{\partial D} \right)^{-1}  \mathcal{S}^{k_0}_{\partial D} \right)
+
\frac{1}{\mu_0} {\mathcal{K}^{k0}_{\partial D}}^* 
- \frac{1}{\mu_1} {\mathcal{K}^{k_1}_{\partial D}}^*   \left( \mathcal{S}^{k_1}_{\partial D} \right)^{-1}  \mathcal{S}^{k_0}_{\partial D}
 \right\} [\psi] \\
=&    \frac{1}{\mu_1}( - \f{1}{2} Id + {\mathcal{K}^{k_1}_{\partial D}}^* ) \circ \left(\mathcal{S}^{k_1}_{\partial D} \right)^{-1} \left[ u_0 \right]  -  \frac{1}{\mu_0}  \f{\p u_0}{\p \nu}   \\
=&   \left( \frac{1}{\mu_1} -  \frac{1}{\mu_0}  \right) \f{\p u_0}{\p \nu}  .
\end{split}
\end{equation}
Similar to our treatment in \cite{ACL} and using \eqref{eq:ss1}, we can now formally write 
\begin{equation}\label{eq:rep1}
\begin{split}
u  - u_0 = & \left( \frac{1}{\mu_1} -  \frac{1}{\mu_0}  \right) \mathcal{S}^{k_0}_{\partial D} \circ \bigg\{   \f{1}{2}  \left( \frac{1}{\mu_0} Id + \frac{1}{\mu_1} \, \left( \mathcal{S}^{k_1}_{\partial D} \right)^{-1}  \mathcal{S}^{k_0}_{\partial D} \right)\\
&+
\frac{1}{\mu_0} {\mathcal{K}^{k0}_{\partial D}}^* 
- \frac{1}{\mu_1} {\mathcal{K}^{k_1}_{\partial D}}^*   \left( \mathcal{S}^{k_1}_{\partial D} \right)^{-1}  \mathcal{S}^{k_0}_{\partial D}
 \bigg\}^{-1}  \left[ \f{\p u_0}{\p \nu}   \right ],
 \end{split}
\end{equation}
when the inverses in the equation do exist. As in \cite{ACL}, we notice that
\beqn
    \mathcal{S}^k_{\partial D} =  \mathcal{S}_{\partial D}  + \omega^2 \, \mathcal{S}^k_{\partial D,-3} \,, \quad 
    {\mathcal{K}^k_{\partial D}}^*  =  {\mathcal{K}_{\partial D}}^* +  \omega^2  \,  \mathcal{K}^k_{\partial D,-3}  \,
\ \text{ and }\ 
   \Lambda_{k_0}  = \Lambda_{0}  +  \omega^2  \,  \Lambda_{k_0,-1}  \,, 
\label{perturbation}
\eqn
where $ \mathcal{K}^k_{\partial D,-3} , \mathcal{S}^k_{\partial D,-3} , \Lambda_{k_0,-1} $ are uniformly bounded w.r.t. $\omega$ and are of order  $-3$, $-3$ and $-1$, respectively.  
With this, one quickly observes that the following lemma holds (cf. \cite{ACL}):
\begin{Lemma} It holds that
\begin{equation}\label{eq:rep1b}
\begin{split}
u  - u_0 =  \mathcal{S}^{k_0}_{\partial D} \circ \left(    \left\{   \lambda(\mu_0^{-1}, \mu_1^{-1}) Id
- {\mathcal{K}_{\partial D}}^*   
 \right\}^{-1}   \circ \Lambda_{k_0} (u_0) +  \omega^2 \, R_{\mu_0,\mu_1,\eps_0,\eps_1,\omega,\partial D, -1}   ( u_0 )  \right),
 \end{split}
\end{equation}
where $R_{\mu_0,\mu_1,\eps_0,\eps_1,\omega,\partial D, -1} $ is uniformly bounded with respect to $\omega \ll  1$ and is of order $-1$.
\end{Lemma}
\noindent

Similar to the static case discussed in Section~\ref{sect:1.2}, for given $\mu_0, \eps_0$ and $\omega\ll 1$, if the following operator equation 
\beqn
\bigg\{   \f{1}{2}  \left( \frac{1}{\mu_0} Id + \frac{1}{\mu_1} \, \left( \mathcal{S}^{k_1}_{\partial D} \right)^{-1}  \mathcal{S}^{k_0}_{\partial D} \right) +
\frac{1}{\mu_0} {\mathcal{K}^{k0}_{\partial D}}^* 
- \frac{1}{\mu_1} {\mathcal{K}^{k_1}_{\partial D}}^*   \left( \mathcal{S}^{k_1}_{\partial D} \right)^{-1}  \mathcal{S}^{k_0}_{\partial D}  \bigg\}  \phi  = 0
\label{resonance}
\eqn
has a non-trivial solution $\phi \in H^{-1/2} ( \partial D, d \sigma)$, then $(\eps_1, \mu_1)$ is said to be a pair of plasmonic eigenvalue and $\phi$ is called a perturbed NP eigenfunction. In this case, the plasmon resonant field in $\mathbb{R}^d\backslash\overline{D}$ is given by $\mathcal{S}_{\partial D}^{k_0}[\phi]$. Next, we consider the geometric properties of the perturbed NP eigenfunctions as well as the associated layer-potentials described above. We quickly realize from \eqref{perturbation} that \eqref{resonance} reads:
\begin{equation}\label{eq:ss3}
\bigg\{  \lambda(\mu_0^{-1}, \mu_1^{-1}) Id - {\mathcal{K}^*_{\partial D}}+  \omega^2 \mathcal{E}_{\mu_0,\mu_1,\eps_0,\eps_1,\omega,\partial D, -3}  \bigg\}  \phi  = 0,
\end{equation}
where $ \mathcal{E}_{\mu_0,\mu_1,\eps_0,\eps_1,\omega,\partial D, -3} $ is uniformly bounded with respect to $\omega \ll  1$ and is of order $-3$. Furthermore, since $ \mathcal{K}_{\partial D}^* +  \omega^2  \mathcal{E}_{\mu_0,\mu_1,\eps_0,\eps_1,\omega,\partial D, -3}$ is compact but not-self adjoint, we have a finite dimensional generalized eigenspace whenever the eigenvalue is non zero \cite{spectral}, (we are unsure of what happens when $\lambda = 0$, i.e. if the kernal of the operator is finite dimensional and if there is a quasi-nilpotent subspace). 
We may therefore consider the following generalized plasmon resonance: find $\phi \in H^{-1/2} ( \partial D, d \sigma)$ such that for some $m \in \mathbb{N} $,
\beqn
\bigg\{   \f{1}{2}  \left( \frac{1}{\mu_0} Id + \frac{1}{\mu_1} \, \left( \mathcal{S}^{k_1}_{\partial D} \right)^{-1}  \mathcal{S}^{k_0}_{\partial D} \right) +
\frac{1}{\mu_0} {\mathcal{K}^{k0}_{\partial D}}^* 
- \frac{1}{\mu_1} {\mathcal{K}^{k_1}_{\partial D}}^*   \left( \mathcal{S}^{k_1}_{\partial D} \right)^{-1}  \mathcal{S}^{k_0}_{\partial D}  \bigg\}^m  \phi  = 0.
\label{generalized_resonance}
\eqn
 It is noted from our earlier discussion that if $( \mu_0,\mu_1,\eps_0,\eps_1,\omega )$ fulfils that \eqref{generalized_resonance} has a solution, then $m$ is finite.

The following lemma characterizes the plasmon resonance when $ \omega \ll  1$.

\begin{Lemma}
\label{close}
Under Assumption (A) and $ \omega \ll  1$, a solution $(  ( \mu_0,\mu_1,\eps_0,\eps_1,\omega ) , m, \phi_{\mu_0,\mu_1,\eps_0,\eps_1,\omega,m} ) $ satisfying the generalized plasmon resonance equation \eqref{generalized_resonance} with a unit $L^2$-norm possesses the following property for all $s \in \mathbb{R}$:
\beqnx
\begin{cases}
\| |D|^{s} \phi_{\mu_0,\mu_1,\eps_0,\eps_1,\omega, m}  - |D|^{s} \phi_i \|_{\mathcal{C}^{0}(\partial D)} &= \mathcal{O}_{i,s}(\omega^2),\medskip\\
 \lambda(\mu_0^{-1}, \mu_1^{-1}) - \lambda_i &= \mathcal{O}_{i}(\omega^2),\\
\end{cases}
\eqnx
for some eigenpair $( \lambda_i , \phi_i )$ of the Neumann-Poincar\'e operator $ {\mathcal{K}^*_{\partial D}}$ with zero wavenumber, and $m \leq m_i$, where $\|\phi\|_{L^2(D)}=1$, and $m$ and $m_i$ signify the algebraic multiplicities of $\lambda$ and $\lambda_i$, respectively. Here the constant in $\mathcal{O}_{i,s}$ depends on both $i$ and $s$, and that in $\mathcal{O}_i$ depends only on $i$.
\end{Lemma}

\begin{proof}
Since the family
$
\{ \mathcal{K}_{\partial D}^* +  \omega^2  \mathcal{E}_{\mu_0,\mu_1,\eps_0,\eps_1,\omega,\partial D, -3} \}_{\omega \geq 0 }
$
is collectively compact, it readily follows from Osborn's Theorem \cite{osborn} and the equivalence of $\| \cdot \|_{H^{-1/2} (\partial D, d \sigma)} $ and $  \| \cdot \|_{L^2_{S_{\partial D}}(\partial D)}$ that a solution $(  ( \mu_0,\mu_1,\eps_0,\eps_1,\omega ) , m, \phi_{\mu_0,\mu_1,\eps_0,\eps_1,\omega,m} ) $ satisfying \eqref{generalized_resonance} also satisfies:
\beqnx
\begin{cases}
\| \phi_{\mu_0,\mu_1,\eps_0,\eps_1,\omega, m}  - \phi_i \|_{H^{-1/2} (\partial D, d \sigma)} &= \mathcal{O}_{i}(\omega^2),\medskip\\
 \lambda(\mu_0^{-1}, \mu_1^{-1}) - \lambda_i &= \mathcal{O}_{i}(\omega^2),\\
\end{cases}
\eqnx
for some eigenpair $( \lambda_i , \phi_i )$ of the Neumann-Poincar\'e operator $ {\mathcal{K}^*_{\partial D}}$.

It remains to obtain the $\| |D|^{s} ( \cdot )  \|_{\mathcal{C}^{0}(\partial D)}$ bounds instead of the $H^{-1/2} (\partial D, d \sigma)$ bounds. For this purpose, let us look into the generalized eigenspace $ E_{ \lambda(\mu_0^{-1}, \mu_1^{-1}) } $ of $ \mathcal{K}_{\partial D}^* +  \omega^2  \mathcal{E}_{\mu_0,\mu_1,\eps_0,\eps_1,\omega,\partial D, -3} $ and pick $\widetilde{\phi_{i,j}} \in E_{ \lambda(\mu_0^{-1}, \mu_1^{-1}) } $ with a unit $H^{-1/2}$-norm satisfying $\widetilde{\phi_{i,m}} =  \phi_{\mu_0,\mu_1,\eps_0,\eps_1,\omega, m} $. Then there exists $\{ \varepsilon_{j,j-1} \}_{j=2}^m $ with $| \varepsilon_{j,j-1}| = O_i (\omega^2) $ such that 
\beqnx
\begin{cases}
 \left( \mathcal{K}_{\partial D}^* +  \omega^2  \mathcal{E}_{\mu_0,\mu_1,\eps_0,\eps_1,\omega,\partial D, -3} - \lambda(\mu_0^{-1}, \mu_1^{-1})  \right)   \widetilde{\phi_{i,j}}   = \varepsilon_{j,j-1}  \widetilde{\phi_{i,j-1}} \text{ for } j = 2,...,m \,,\\
\left( \mathcal{K}_{\partial D}^* +  \omega^2  \mathcal{E}_{\mu_0,\mu_1,\eps_0,\eps_1,\omega,\partial D, -3} - \lambda(\mu_0^{-1}, \mu_1^{-1})  \right)   \widetilde{\phi_{i,1}}   = 0 \,,
\end{cases}
\eqnx
which can always be done by rescaling the basis giving the Jordan block representation with a scaling factor of $1/\varepsilon_{j,j-1} $.
Then Osborn's Theorem and the equivalence of norms yield
\beqnx
\| \widetilde{\phi_{i,j}}   - \phi_{i,j} \|_{H^{-1/2} (\partial D, d \sigma)} &= \mathcal{O}_{i}(\omega^2)
\eqnx
for some $\phi_{i,j} \in \mathcal{C}^{\infty} (\partial D) $ sitting in the eigenspace of $\mathcal{K}_{\partial D}^*$. 
Taking the difference between the system in the generalized eigenspace and the original eigenvalue equations:
\beqnx
\begin{cases}
 & \left( \mathcal{K}_{\partial D}^* +  \omega^2  \mathcal{E}_{\mu_0,\mu_1,\eps_0,\eps_1,\omega,\partial D, -3}  \right)   ( \phi_{i,j} - \widetilde{\phi_{i,j}}   )  - \omega^2  \mathcal{E}_{\mu_0,\mu_1,\eps_0,\eps_1,\omega,\partial D, -3} \phi_{i,j} \\ 
& = \lambda(\mu_0^{-1}, \mu_1^{-1})   ( \phi_{i,j} - \widetilde{\phi_{i,j}}   )  + ( \lambda_i  -\lambda(\mu_0^{-1}, \mu_1^{-1}) ) \phi_{i,j} - \varepsilon_{j,j-1}  \widetilde{\phi_{i,j-1}}   \text{ for } j = 2,...,m \,,\medskip\\
 & \left( \mathcal{K}_{\partial D}^* +  \omega^2  \mathcal{E}_{\mu_0,\mu_1,\eps_0,\eps_1,\omega,\partial D, -3}  \right)   ( \phi_{i,1} - \widetilde{\phi_{i,1}}   )  - \omega^2  \mathcal{E}_{\mu_0,\mu_1,\eps_0,\eps_1,\omega,\partial D, -3} \phi_{i,1} \\ 
& = \lambda(\mu_0^{-1}, \mu_1^{-1})   ( \phi_{i,1} - \widetilde{\phi_{i,1}}   )  + ( \lambda_i  -\lambda(\mu_0^{-1}, \mu_1^{-1}) ) \phi_{i,1} .
\end{cases}
\eqnx
Now, under Assumption (A), $ \mathcal{K}_{\partial D}^* +  \omega^2  \mathcal{E}_{\mu_0,\mu_1,\eps_0,\eps_1,\omega,\partial D, -3} : H^s (\partial D) \rightarrow H^{s+1} (\partial D) $, and therefore we have from the above system that 
\beqnx
\|   \widetilde{\phi_{i,j}}  -  \phi_{i,j}  \|_{H^{1/2} (\partial D, d \sigma) }  = \mathcal{O}_i( \omega^2),
\eqnx
for all $j = 1,...,m$ with a different constant.  One now gazes at the above system.  Together with a bootstrapping argument and the fact that Assumption (A) gives $\phi_{i,j} \in \mathcal{C}^{\infty} (\partial D)$, we arrive at, for all $l \in \mathbb{R} $,
\beqnx
\|   \widetilde{\phi_{i,j}}  -  \phi_{i,j}  \|_{H^l (\partial D, d \sigma) }  = \mathcal{O}_{i,l}( \omega^2) \,.
\eqnx
Our conclusion follows after applying the Sobelov embedding theorem to bound the $\| |D|^{s} (\cdot )  \|_{\mathcal{C}^{0}(\partial D)}$ semi-norm by the $H^{s+l} (\partial D, d\sigma)$ norm for large enough $l$ .

The proof is complete. 
\end{proof}

We aim to know whether the generalized plasmon resonance (cf. \eqref{generalized_resonance}) always exists when $ \omega \ll  1$.  The following lemma addresses this issue.

\begin{Lemma}
\label{existence}
Given any non-zero $\lambda_i \in \sigma( \mathcal{K}_{\partial D}^* )$, the spectrum of $ \mathcal{K}_{\partial D}^* $, for any $ (\tilde{\mu}_0, \tilde{\mu}_1) \in  D_i := \{ (\mu_0, \mu_1) \in \mathbb{C}^2\backslash \{(0,0)\} \,: \, \lambda(\tilde{\mu}_0^{-1}, \tilde{\mu}_1^{-1} ) = \lambda_i \,, \, \mu_0 - \mu_1 \neq 0 \,  \} $ (which is non-empty), there exists $ 0 < \omega_i \ll  1$ such that for all $\omega < \omega_i$, the set 
\beqnx
 &&\bigg\{ ( \mu_0,\mu_1,\eps_0,\eps_1 ) \in  \mathbb{C}^2 \backslash\{\mu_0 - \mu_1 = 0 \} \times ( \mathbb{C} \backslash \mathbb{R}^+ )^2 ; \\
&& \text{ there exists } m \in \mathbb{N}, \phi \in H^{-1/2} (\partial D, \sigma)  \text{ such that } (  ( \mu_0,\mu_1,\eps_0,\eps_1,\omega ) , \phi, m) \text { satisfies } \eqref{generalized_resonance} \bigg\}
\eqnx
forms a complex co-dimension $1$ surface in a neighborhood of $(\tilde{\mu}_0, \tilde{\mu}_1)$.
\end{Lemma}

\begin{proof}
Given a non-zero $\lambda_i \in \sigma( \mathcal{K}_{\partial D}^* )$, we consider a function $F_i$ defined over $\partial D$ and $\lambda_i \in \sigma( \mathcal{K}_{\partial D}^* )$.
In particular, by Osborn's Theorem \cite{osborn} and the smooth dependence of $\mathcal{E}$ on $(\mu_0,\mu_1,\eps_0,\eps_1,\omega )$, there exists $0 < \tilde{\omega}_i \ll  1$ (depending on $i$) such that we have a (non-unique) smooth choice of function:
\beqnx
F_{i,\delta} :  \mathbb{C}^2 \backslash\{\mu_0 - \mu_1 = 0 \} \times ( \mathbb{C} \backslash \mathbb{R}^+ )^2 \times (0, \tilde{ \omega_i } ) & \rightarrow & \mathbb{C}, \\
F_i (\mu_0, \mu_1, \varepsilon_0, \varepsilon_1, \omega) &=&  \widetilde{\lambda}_{i}  (\mu_0, \mu_1, \varepsilon_0, \varepsilon_1, \omega) -  \lambda(\mu_0^{-1}, \mu_1^{-2} ) ,
\eqnx
where
\beqnx
\tilde{\lambda}_{i}  (\mu_0, \mu_1, \varepsilon_0, \varepsilon_1, \omega) \in \sigma\left( \mathcal{K}_{\partial D}^* +  \omega^2  \mathcal{E}_{\mu_0,\mu_1,\eps_0,\eps_1,\omega,\partial D, -3} \right)
\eqnx
is such that 
\[
\lim_{\varepsilon \rightarrow 0} \tilde{\lambda}_{i}  (\mu_0, \mu_1, \varepsilon_0, \varepsilon_1, \omega)  = \lambda_i \,.
\]
We now note that, for any $(\tilde{\mu}_0, \tilde{\mu}_1) \in D_i$ in this set,
\[
F_i (\tilde{\mu}_0, \tilde{\mu}_1, \varepsilon_0, \varepsilon_1, 0) = 0,
\]
for all $\varepsilon_0, \varepsilon_1$ in the domain of the function.
Moreover, we can directly verify that
\[
\partial_{\omega}  F_i (\tilde{\mu}_0,\tilde{\mu}_1, \varepsilon_0, \varepsilon_1, 0)  = 0 \,, \quad  \partial_{\varepsilon_0, \varepsilon_1}  F_i (\tilde{\mu}_0,\tilde{\mu}_1, \varepsilon_0, \varepsilon_1, 0)  = 0 ,
\]
whereas
\[
\partial_{\mu_0,\mu_1} F_i (\tilde{\mu}_0,\tilde{\mu}_1, \varepsilon_0, \varepsilon_1, 0) = \partial_{\mu_0,\mu_1}   \lambda(\tilde{\mu}_0^{-1}, \tilde{\mu}_1^{-1} )    = \left( -\f{\tilde\mu_1}{2(\tilde\mu_1-\tilde\mu_0)^2}, \f{\tilde \mu_0}{2(\tilde\mu_1-\tilde\mu_0)^2} \right).
\]
Hence, we have
\[
 \partial_{\mu_0} F_i (\tilde{\mu}_0,\tilde{\mu}_1, \varepsilon_0, \varepsilon_1, 0) \neq 0 \quad \text{ or } \quad  \partial_{\mu_1} F_i (\tilde{\mu}_0,\tilde{\mu}_1, \varepsilon_0, \varepsilon_1, 0) \neq 0\,.
\]
Therefore, applying the inverse function theorem in a neighborhood of any chosen point in $ D_i  \times ( \mathbb{C} \backslash \mathbb{R}^+ )^2 \times \{ 0\}  $, we obtain either a unique smooth function $l_0: B_{\delta}(\tilde{\mu}_0) \rightarrow l_0 (B_{\delta}(\tilde{\mu}_0))$ fulfilling 
\[
F_i (\mu_0, l_{0,1}(\mu_0), l_{0,2}(\mu_0), l_{0,3}(\mu_0), l_{0,4}(\mu_0))  = 0,
\]
or a unique smooth function $l_1: B_{\delta}(\tilde{\mu}_1) \rightarrow l_0 (B_{\delta}(\tilde{\mu}_1))$ fulfilling
\[
F_i (l_{1,1}(\mu_1), \mu_1, l_{1,2}(\mu_1), l_{1,3}(\mu_1), l_{1,4}(\mu_1))  = 0. 
\]
If we obtain $l_0$, let us take  $ \omega_i \leq \tilde{\omega_i} $ such that $ \omega_i \in l_{0,3} (B_{\delta}(\tilde{\mu}_0))$.
Otherwise, we take $ \omega_i \leq \tilde{\omega}_i $ such that $ \omega_i \in l_{1,3} (B_{\delta}(\tilde{\mu}_1))$. The conclusion stated in the lemma readily follows.  
\end{proof}

By Lemma~\ref{existence}, we easily see that there are infinitely many choices of $(\varepsilon_1, \mu_1)$ such that the (genearalized) plasmon resonance occurs around $\lambda_i$. Combining this with a similar perturbation argument as in the proof of Lemma~\ref{existence}, our conclusions of the plasmon resonance in the electrostatic case transfers to the Helmholtz transmission problem to show concentration of plasmon resonances at high-curvature points.  For instance, we have the following result.
\begin{Theorem}
\label{theorem4}
Given any $x \in \partial D$, let us consider $ \{ \chi_{x,\delta} \}_{\delta > 0}$ being a family of smooth nonnegative bump functions compactly supported in $B_{\delta} (x)$ with $ \int_{\partial D} \chi_{p,\delta} \, d \sigma = 1$.
Under Assumption (A), given $r \leq s$, $\alpha\in\mathbb{R}$ and $p, q \in \partial D$,
we have a choice of $\delta(h)$ and $\omega(h) $ both depending on $r,s,p,q$ and $\alpha$ such that for any $\omega <  \omega (h)$, there exists 
$$\left( ( \mu_{0,i},\mu_{1,i},\eps_{0,i},\eps_{1,i}, \omega ), m_i , \phi_{ \mu_{0,i},\mu_{1,i},\eps_{0,i},\eps_{1,i}, \omega, m_i }  \right) $$ solving \eqref{generalized_resonance}, and as $h\rightarrow +0$, we have $\delta(h) \rightarrow 0$,  $\omega(h) \rightarrow 0$ and
{
\beqn
\begin{split}
& \frac{
\sum_{r h \leq \lambda^2_i  \leq s h} c_i \,
 \int_{\partial D} \chi_{p,\delta (h) }(x) | |D|^{\alpha}  \phi_{ \mu_{0,i},\mu_{1,i},\eps_{0,i},\eps_{1,i}, \omega, m_i }  (x) |^2 d \sigma (x) 
}
{
\sum_{r h \leq \lambda^2_i  \leq s h} c_i \,
 \int_{\partial D} \chi_{q,\delta (h) }(x) | |D|^{\alpha}  \phi_{ \mu_{0,i},\mu_{1,i},\eps_{0,i},\eps_{1,i}, \omega, m_i }  (x) |^2 d \sigma (x) 
} \\
=&  \frac{ \int_{ \{ H (p, \cdot) =1 \} }  | \xi |_{g(p)}^{1+ 2 \alpha }  d \sigma_{p,H}  }{ \int_{ \{ H (q, \cdot) =1 \} }  | \xi |_{g(q)}^{1+ 2 \alpha }   d \sigma_{q,H}  } + o_{r,s,p,q,\alpha}(1), 
\label{concentration11}
\end{split}
\eqn
}where $c_i := | \phi_i |_{H^{-\frac{1}{2}}(\partial D , d \sigma)}^{-2} $.  Here, the little-$o$ depends on $r,s,p,q$ and $\alpha$.
\end{Theorem}
\begin{proof}
From Theorem \ref{theorem1},  we have a choice of $\delta(h)$ depending on $r,s,p,q$ and $\alpha$ such that,
for any given $\varepsilon>0$, there exists $h_0$ depending on $r,s,p,q,\alpha$ such that for all $h < h_0$,
\beqn
\left| \frac{ \sum_{r \leq \lambda_i^2 (h) \leq s} c_i  \int_{\partial D} \chi_{p,\delta (h) }(x) | | D |^{\alpha}   \, \phi_i (x) |^2 d \sigma }{ \sum_{r \leq \lambda_i^2 (h) \leq s} c_i  \int_{\partial D} \chi_{q,\delta (h) }(x) | | D |^{\alpha}   \, \phi_i (x) |^2 d \sigma } -  \frac{ \int_{ \{ H (p, \cdot) =1 \} }  | \xi |_{g(p)}^{1+ 2 \alpha }  d \sigma_{p,H}  }{ \int_{ \{ H (q, \cdot) =1 \} }  | \xi |_{g(q)}^{1+ 2 \alpha }   d \sigma_{q,H}  }  \right| \leq \varepsilon.
\label{ineq11a}
\eqn
For each $h < h_0$, from Lemma \ref{existence}, there exists $\tilde{\omega}(h) := \min_{\{ i \in \mathbb{N}:  r h \leq \lambda^2_i  \leq s h\} }\{ \omega_i \}$ such that for all $\omega < \omega (h)$, there exists 
$$\left( ( \mu_{0,i},\mu_{1,i},\eps_{0,i},\eps_{1,i}, \omega ), m_i , \phi_{ \mu_{0,i},\mu_{1,i},\eps_{0,i},\eps_{1,i}, \omega, m_i }  \right) $$ solving \eqref{generalized_resonance}.   By Lemma \ref{close}, upon a rescaling of $\phi_{ \mu_{0,i},\mu_{1,i},\eps_{0,i},\eps_{1,i}, \omega, m_i } $ while still denoting it as $\phi_{ \mu_{0,i},\mu_{1,i},\eps_{0,i},\eps_{1,i}, \omega, m_i } $, we have
\beqnx
\| |D|^{\alpha} \phi_{ \mu_{0,i},\mu_{1,i},\eps_{0,i},\eps_{1,i}, \omega, m_i }  - |D|^{\alpha} \phi_i \|_{\mathcal{C}^{0 }(\partial D)} \leq  C_{i,\alpha} \omega^2.
\eqnx
In particular, we can make a smaller choice of  $\omega(h) < \tilde{\omega}(h) $ depending on $r,s,p,q,\alpha$ such that for all $\omega < \omega(h)$,we have
{\small
\beqnx
& &\| | |D|^{\alpha} \phi_{ \mu_{0,i},\mu_{1,i},\eps_{0,i},\eps_{1,i}, \omega, m_i } |^2  - | |D|^{\alpha} \phi_i |^2 \|_{\mathcal{C}^{0 }(\partial D)}  \\
&\leq & 10^{-2} \varepsilon / \sum_{r h \leq \lambda^2_i  \leq s h} c_i / \min \left\{1,  \min_{y = p,q} \left\{  \left( \sum_{r \leq \lambda_i^2 (h) \leq s} c_i  \int_{\partial D} \chi_{y,\delta (h) }(x) | | D |^{\alpha}   \, \phi_i (x) |^2 d \sigma  \right)^{-2} \right \}  \right \} \,.
\eqnx
}Therefore, with this choice of $\omega(h) $, we have, for all $\omega < \omega(h)$,

\begin{equation}\label{ineq11b}
\begin{split}
&\Bigg|
\frac{
\sum_{r h \leq \lambda^2_i  \leq s h} c_i \,
 \int_{\partial D} \chi_{p,\delta (h) }(x) | |D|^{\alpha}  \phi_{ \mu_{0,i},\mu_{1,i},\eps_{0,i},\eps_{1,i}, \omega, m_i }  (x) |^2 d \sigma (x) 
}
{
\sum_{r h \leq \lambda^2_i  \leq s h} c_i \,
 \int_{\partial D} \chi_{q,\delta (h) }(x) | |D|^{\alpha}  \phi_{ \mu_{0,i},\mu_{1,i},\eps_{0,i},\eps_{1,i}, \omega, m_i }  (x) |^2 d \sigma (x) 
}\\
& \qquad\qquad\qquad-
\frac{ \sum_{r \leq \lambda_i^2 (h) \leq s} c_i  \int_{\partial D} \chi_{p,\delta (h) }(x) | | D |^{\alpha}   \, \phi_i (x) |^2 d \sigma }{ \sum_{r \leq \lambda_i^2 (h) \leq s} c_i  \int_{\partial D} \chi_{q,\delta (h) }(x) | | D |^{\alpha}   \, \phi_i (x) |^2 d \sigma } \Bigg| \leq \varepsilon.
\end{split}
\end{equation}
Combining \eqref{ineq11b} with \eqref{ineq11a} readily yields our conclusion.

The proof is complete. 
\end{proof}

Likewise we obtain the following result.
\begin{Theorem}
\label{theorem5}
Under Assumption (A), given $r,s$, there exists $S(h) \subset J(h) := \{i \in \mathbb{N} : r h \leq \lambda^2_i  \leq s h  \}$ and $\omega(h) $ such that, for all $\varphi \in \mathcal{C}^\infty ( \partial D)$, we have for any $\omega < \omega (h)$, there exists 
$$\left( ( \mu_{0,i},\mu_{1,i},\eps_{0,i},\eps_{1,i}, \omega ), m_i , \phi_{ \mu_{0,i},\mu_{1,i},\eps_{0,i},\eps_{1,i}, \omega, m_i }  \right) $$ solving \eqref{generalized_resonance}, such that as $h\rightarrow+0$, we have $\omega(h) \rightarrow 0$ and
\begin{equation}\label{eq:tt1}
\begin{split}
& \max_{i \in S(h)}  \bigg|   \int_{ \partial D}  \varphi(x) \bigg( c_i \, | |D|^{-\frac{1}{2}} \phi_{ \mu_{0,i},\mu_{1,i},\eps_{0,i},\eps_{1,i}, \omega, m_i }  (x) |^2\\
& \hspace*{2cm} - \int_{M_{X_{H},\text{erg}}(\{ H = 1\})}  \mu (x)  g_i (\mu) d \nu(\mu)  \bigg) d \sigma(x)
 \bigg|= o_{r,s}(1)  \,.
 \end{split}
\end{equation}
Here, $S(h)$, $\{ g_{i} : M_{X_{H},\text{erg}}(\{ H = 1\}) \rightarrow \mathbb{C}  \}_{i \in \mathbb{N} }$ and $ \mu(p)$ are described as in Theorem~\ref{theorem2}.
In particular, we remind that
\beqnx
 \frac{ \int_{M_{X_{H},\text{erg}}(\{ H = 1\})}   \mu (p)  d \nu(\mu) }{ \int_{M_{X_{H},\text{erg}}(\{ H = 1\})}   \mu (q)  d \nu(\mu) }  =  \frac{ \int_{\{ H(p,\cdot)=1\}} d \sigma_{p,H} }{ \int_{\{ H(q,\cdot)=1\}} d \sigma_{q,H} } \text{ a.e. } (d \sigma \otimes d \sigma) (p,q) \,.
\eqnx
If the Hamiltonian flow given by $X_H$ is ergodic with respect to $\sigma_H$ on $\{H = 1\}$, then 
\beqnx
\max_{i \in S(h)}  \left|   \int_{ \partial D}  \varphi(x) \left( c_i \, | |D|^{-\frac{1}{2}} \phi_{ \mu_{0,i},\mu_{1,i},\eps_{0,i},\eps_{1,i}, \omega, m_i }  (x) |^2  -  \frac{ \int_{\{ H(x,\cdot)=1\}} d \sigma_{x,H}  }{ \sigma_H(\{H =1\})}   \right) d \sigma(x)
 \right|= o_{r,s}(1)  \,.
\eqnx
\end{Theorem}

\begin{proof}
Let $r,s $ be given.  Consider $\varphi \in \mathcal{C}^{\infty}(\partial D)$.
Given $\eps > 0$, by Theorem \ref{theorem2} and considering $h_0$ small enough such that for all $ h < h_0$, we have
\beqnx
\max_{i \in S(h)}  \left|   \int_{ \partial D}  \varphi(x) \left( c_i \, |D|^{-\frac{1}{2}} \phi_i (x) |^2 - \int_{M_{X_{H},\text{erg}}(\{ H = 1\})}  \mu (x)  g_i (\mu) d \nu(\mu)  \right) d \sigma(x)
 \right| \leq \eps .
\eqnx
Now, for each $h < h_0$, from Lemma \ref{existence}, there exists $\tilde{\omega}(h) = \min \left\{ \min_{i \in  S(h)} \omega_i, 1 \right\}$ such that for all $\omega < \tilde{ \omega }(h)$, there exists 
$$\left( ( \mu_{0,i},\mu_{1,i},\eps_{0,i},\eps_{1,i}, \omega ), m_i , \phi_{ \mu_{0,i},\mu_{1,i},\eps_{0,i},\eps_{1,i}, \omega, m_i }  \right) $$ solving \eqref{generalized_resonance}.   By Lemma \ref{close}, again upon a rescaling of $\phi_{ \mu_{0,i},\mu_{1,i},\eps_{0,i},\eps_{1,i}, \omega, m_i } $ while still denoting it as $\phi_{ \mu_{0,i},\mu_{1,i},\eps_{0,i},\eps_{1,i}, \omega, m_i } $, we have
\begin{equation}\label{eq:ff1}
\begin{split}
& \max_{i \in S(h)} c_i  \left|   \int_{ \partial D}  \varphi(x) \left(  \, | |D|^{-\frac{1}{2}} \phi_i (x) |^2 -  \, | |D|^{-\frac{1}{2}} \phi_{ \mu_{0,i},\mu_{1,i},\eps_{0,i},\eps_{1,i}, \omega, m_i }  (x) |^2  \right) d \sigma(x)
 \right|\\
 & \leq C_{S(h)} \| \varphi \|_{\mathcal{C}^{0}(\partial D ) } \, \omega^2.
 \end{split}
\end{equation}
We may now choose 
$$ \omega(h) \leq \min \left\{ \eps,  \tilde{ \omega }(h),  \tilde{\omega}(h) / C_{S(h)} \right\} \, .$$
Then for all $\omega < \omega(h)$, we finally have from \eqref{eq:ff1} and Corollary \ref{corollary3} that
\[
\begin{split}
&  \max_{i \in S(h)}  \bigg|   \int_{ \partial D}  \varphi(x) \bigg( c_i \, | |D|^{-\frac{1}{2}} \phi_{ \mu_{0,i},\mu_{1,i},\eps_{0,i},\eps_{1,i}, \omega, m_i }  (x) |^2\\
& \hspace*{2cm} - \int_{M_{X_{H},\text{erg}}(\{ H = 1\})}  \mu (x)  g_i (\mu) d \nu(\mu)  \bigg) d \sigma(x)
 \bigg| \\
 \leq & \left( 1+ \| \varphi \|_{\mathcal{C}^{0}(\partial D ) } \right)  \eps  \, .
\end{split}
\]

The proof is complete. 
\end{proof}


\section*{Acknowledgements}
The work of H Liu was supported by the startup grant from City University of Hong Kong, Hong Kong RGC General Research Funds, 12301218, 12302919 and 12301420.

\appendix
\section{Further remarks upon some geometric behaviors of the Hamiltonian flow} \label{Appendix_A}

In this appendix, we would like to briefly explore some geometric behaviors of the Hamiltonian flow, which should help to gain better understanding of $ {M_{X_{H},\text{erg}}(\{ H = 1\})}  $ as well as translate to the understanding of the NP eigenfunctions. 
We would also like to explore Assumption (A), which is equivalent to $\overline{ \{H = 1\} } \bigcap \left(  \partial D  \times \{0\} \right)  = \emptyset$, and has been imposed in our study up till now.
First, we have the following elementary property of $X_{H}$:
\begin{Lemma}
\label{lemma_well}
If $\xi \neq 0$, then $X_{H} \neq 0$.
\end{Lemma}
\begin{proof}
Consider $S := \log H$.  Since $X_{H}$ preserves $H$, let us consider its action only on $\{H = c \}$ with $c \neq 0$.  By choosing a local coordinate, one can directly compute that $$ \partial_{\xi} H = H \partial_{\xi} S = 2 c \left( 2 \frac{ (d-1) \mathcal{H}(x) g^{-1}(x) -  g^{-1}(x) \mathcal{A}(x)  g^{-1}(x) \xi }{ \left|  \langle (d-1) \mathcal{H}(x) g^{-1}(x) -  g^{-1}(x) \mathcal{A}(x)  g^{-1}(x) \xi, \xi \rangle \right|  }  - 3 \frac{g^{-1} \xi}{\langle g^{-1}(x) \xi,\xi \rangle } \right) \,.$$
Therefore we immediately infer that $ \langle \partial_{\xi} H , \xi \rangle = 2(\pm 2-3)c $, which ensures that $ \partial_{\xi} H  \neq 0$, and hence our conclusion holds.
\end{proof}

\subsection{The non-singular case when $\overline{ \{H = 1\} } \bigcap \left(  \partial D  \times \{0\} \right)  = \emptyset$ }

Let us assume $\overline{ \{H = 1\} }$ does not contain $(x ,0) \in \partial D \times \{0\} \hookrightarrow T^*(\partial D ) $.  We would like to look into the local property of the flow.  As an example, we only consider $d =3$.  In fact, the condition $\overline{ \{H = 1\} } \bigcap \left(  \partial D \times \{0\} \right)  = \emptyset$ readily implies that the Gaussian curvature $\kappa(x) := \kappa_1(x)\, \kappa_2(x) \neq 0$. 
By the compactness of the surface, we have $\kappa(x) >0 $ for some $x \in \partial D$. Then by continuity we have $\kappa(x) > c $ for all $x \in  \partial D $ for some $c > 0$. An application of the Gauss-Bonnet theorem readily yields the Euler characteristic of the surface $\chi( \partial D) >  0$, and hence $\partial D$ is diffeomorphic to a sphere.  Moreover, there exists $c_0$ such that the matrix $\mathcal{A}(x) > c_0 \, Id $ for all $x \in \partial D$, i.e., the domain $D$ is strictly convex.
The following figure shows a typical example of $ \{ H(x ,\cdot) = 1 \} $.

\begin{figurehere} \centering
\includegraphics[width=4cm,height=2.5cm]{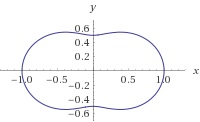}
\caption{Level curve $\overline{ \{ H(a , \xi) = 1 \} } $ for a fixed $a \in \partial D $ when $A(a) = \text{diag} (1, 0.5) $, where $\xi = (x,y) $.} 
\end{figurehere}
\noindent In this case, locally around a point $ (x , \xi ) \in \{H = 1\} $, the flow $ X_H$ given by $\partial_t (x(t),p(t)) = X_{H} (x(t), p(t))$ projects to the $x$-coordinate to give $\partial_t x(t) = \partial_{\xi} H $, which is the normal of the level set $ \{ H(x ,\cdot) = 1 \}$.

In general when $d > 2$, we suppose that $X_{H}$ generates a Hamiltonian circle action (i.e. $\mathbb{T}^1$ action) over $T^* (\partial D)$.  Then by Lemma \ref{lemma_well}, the circle action has no critical point on $ \{H = 1\} = \overline{ \{H = 1\} }$, and that $1$ is a regular value of $H$.  We may therefore perform a symplectic reduction to obtain $ M :=  \{H = 1\} /  \mathbb{T}^1 $. From the compactness of $\{H = 1\}$, $M$ is now a compact symplectic manifold of dimension $2d-4$.  $M$ also provides a parametrization of the set of periodic orbits (equivalently, of the ergodic measures in this case).  In this case, we can appeal to results of classical symplectic geometry to classify the global structure of the flow.

We would like to remark that the non-singular case is rather restrictive, e.g. in $d=3$, any $ \partial \Omega $ not diffeomorphic to $\mathbb{S}^2$ would admit a flow $X_H$ on $\overline{\{ H = 1\}}$ with singularities.

\subsection{The singular case when $\overline{ \{H = 1\} } \bigcap \left(  \partial D \times \{0\} \right) \neq \emptyset$ } 

We can also consider the dynamics where $\overline{ \{H = 1\} }$ may contain points $(x ,0)  \in  \partial D \times \{0\}  $.  
The critical set may now be highly singular and degenerate. 
To appropriately (and mildly) resolve the singularity of $H$ and $X_H$ around $(x,0)$, we consider 
\[
\tilde{\tilde{H}}(x,\xi) := \arctan( H(x,\xi)) = \arctan( \exp ( S(x,\xi) ) )  \,.
\]
One directly verifies that $\tilde{\tilde{H}}$ removes the $(-2)$-order singularity of $H$  (which is smooth away from zero) at $\xi = 0$ in the following sense: that $\tilde{\tilde{H}}$ is now furthermore bounded, and directionally differentiable at $0$, all the while generating a rescaled flow of $X_H$ away from the singularities. 
In particular, one quickly checks that
\beqnx
\partial \tilde{\tilde{H}} = \frac{H}{1+ H^2} \, \partial S = 0 \quad \Leftrightarrow \quad \xi = 0  \,.
\eqnx
Therefore we have the following lemma.
\begin{Lemma}
\label{lemma_label}
$X_{\tilde{\tilde{H}}} = 0$ if and only if $\xi = 0$.  
\end{Lemma}
The set of critical points $\{X_{\tilde{\tilde{H}}} = 0\} = \partial D \times \{0\} $ are still highly degenerate and very singular. 
As an illustrating example, let us take $d= 3$.  When $d=3$, one of the cases that $\overline{ \{H = 1\} } \bigcap \left(  \partial D  \times \{0\} \right) \neq \emptyset$ is when $\overline{ \{H = 1\} } $ contains a point $(x,0)$ with its mean curvature $H(x) = 0$.  Near such a point $x \in \partial D $, we write $\lambda(x) := \kappa_1(x) = - \kappa_2(x)$ and $\xi = r \omega = r (\cos(\theta), \sin(\theta) ) $. Then $\overline { \{ H(x , r \omega) = 1 \} } $ can be parametrized by
\[ r^2(\theta) = \lambda^2 (x) \cos^2 ( 2 \theta ) \, . \]  
The following figure shows $\overline { \{ H(p ,\cdot) = 1 \} }$ in this degenerate and singular case.

\begin{figurehere} \centering
\includegraphics[width=4cm,height=4cm]{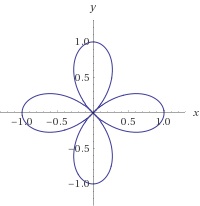}
\caption{The closure of the level curve $\overline{ \{ H(a , \xi) = 1 \} } $ for a fixed $a \in \partial \Omega$ when $A(a) = \text{diag} (1, -1) $, where $\xi = (x,y) $.} 
\end{figurehere}
\noindent 
\noindent Locally around a point $ (x , \xi ) \in \overline{ \{H = 1\} } $ away from $ (x,0)$ where $H(x) = 0$, the flow $ X_{\tilde{\tilde{H}}}$ is again given by $\partial_t (x(t),p(t)) = X_{\tilde{\tilde{H}}} (x(t), p(t))$. It projects to the $x$-coordinate to give $\partial_t x(t) = \partial_{\xi} \tilde{\tilde{H}} $, which is the normal of the level set $ \{ H(x ,\cdot) = 1 \}$.  However, when $ (x , \xi ) \in \overline{ \{H = 1\} } $ is close to $ (x,0)$ where $H(x) = 0$, we can see from the above figure that the normal of the level set $ \overline{ \{ H(x ,\cdot) = 1 \} }$ is behaving pathologically, creating a pathological behavior of the flow around that point.
Further study of the local and global structures of the flow $X_{\tilde{\tilde{H}}}$ (e.g. its dynamical property) will be the subject of a forthcoming work.


\begin{thebibliography}{99}





%
%
\bibitem{resol1} { H. Ammari, Y.T. Chow, and J. Zou}, {\em Super-resolution in imaging high contrast targets from the perspective of scattering coefficients},  J. Math. Pures Appl. {\bf 111} 191-226, 2018.
%

\bibitem{ACKLM} H. Ammari, G. Ciraolo, H. Kang, H. Lee, and G. Milton, {\it Spectral analysis of a Neumann-Poincar\'e-type operator and analysis of cloaking due to anomalous localized resonance}, Arch. Ration. Mech. Anal., {\bf 208} (2013), 667--692.

\bibitem{ACKLM2} H. Ammari, G. Ciraolo, H. Kang, H. Lee and G. W. Milton, {\it Anomalous localized resonance using a
folded geometry in three dimensions}, Proc. R. Soc. A, \textbf{469} (2013), 20130048. 

\bibitem{ACL} H. Ammari, Y.T. Chow and H. Liu, {\it Localized sensitivity analysis at high-curvature boundary points of reconstructing inclusions in transmission problems}, arXiv:1911.00820




\bibitem{book}
{H. Ammari and H. Kang}, {\em Polarization and Moment
    Tensors: With Applications to Inverse Problems and Effective
    Medium Theory}, Applied Mathematical Sciences 162, Springer-Verlag, New York, 2007.





\bibitem{Amm1} H. Ammari, P. Millien, M. Ruiz and H. Zhang, {\it Mathematical analysis of plasmonic nanoparticles: the
scalar case}, Arch. Rational Mech. Anal., \textbf{224} (2017), 597--658.

\bibitem{AJ} K. Ando, Y.-G. Ji, H. Kang, H. Hyeonbae, D. Kawagoe and Y. Miyanishi, {\it Spectral structure of the Neumann-Poincar\'e operator on tori},
 Ann. Inst. H. Poincar\'e Anal. Non Lin\'eaire, \textbf{36} (2019),1817--1828. 

\bibitem{ando}  K. Ando and H. Kang, {\it Analysis of plasmon resonance on smooth domains using spectral properties of the Neumann-Poincar\'e operator}, J. Math. Anal. Appl., \textbf{435} (2016),  162--178. 

\bibitem{AKL} K. Ando, H. Kang and H., Liu, {\it Plasmon resonance with finite frequencies: a validation of the quasi-static
approximation for diametrically small inclusions}, SIAM J. Appl. Math., \textbf{76} (2016), 731--749.

\bibitem{BS} D. J. Bergman and M. I. Stockman, {\it Surface plasmon amplification by stimulated emission of radiation:
quantum generation of coherent surface plasmons in nanosystems}, Phys. Rev. Lett., \textbf{90} (2003), 027402.

\bibitem{birkhoff}
G. D. Birkhoff,
\textit{Proof of the Ergodic Theorem},
Proceedings USA Academy \textbf{17} (1931), 656 - -660.




\bibitem{curv_Liu_3} { E. Bl{\aa}sten, H. Li, H. Liu and Y. Wang}, {\em Localization and geometrization in plasmon resonances and geometric structures of Neumann-Poincar\'e eigenfunctions}, ESAIM: Math. Model. Numer. Anal., \textbf{54} (2020), no. 3, 957--976.


\bibitem{BZ} E. Bonnetier and H. Zhang, {\it Characterization of the essential spectrum of the Neumann-Poincar\'e operator in 2D domains with corner via Weyl sequences}, Rev. Mat. Iber. \textbf{35} (2019), 925--948. 



\bibitem{erg2} Y. Colin de Verdi\`ere, {\it Ergodicit\'e et functions propres du Laplacien}, Commun. Math. Phys. \textbf{102} (1985), 497--502.

\bibitem{spectral}
J.B. Conway, A Course in Functional Analysis, Graduate Texts in Mathematics 96, Springer, 1990.


\bibitem{Ego}  Ju. V. Egorov, {\it The canonical transformations of pseudodifferential operators}, (Russian), Uspehi Mat. Nauk, 24 (1969), 235--236.

\bibitem{FM} D. R. Fredkin and I. D. Mayergoyz, {\it Resonant behavior of dielectric objects (electrostatic resonances)}, 
Phys. Rev. Lett., \textbf{91} (2003), 253902.

\bibitem{erg_b} P. Gerard and E. Leichtnam, {\it Ergodic properties of eigenfunctions for the Dirichlet problem}, Duke Math J. \textbf{71} (1993), 559--607.

\bibitem{homogeneous}
L. Grafakos, R.H. Torres, {\it Pseudodifferential operators with homogeneous symbols}, Michigan Math. J, \textbf{46}(1999), p.p. 261-269.

\bibitem{G} D. Grieser, {\it The plasmonic eigenvalue problem}, Rev. Math. Phys., \textbf{26} (2014), 1450005.

\bibitem{erg4} B. Helffer, A. Martinez, and D. Robert, {\it Ergodicite et limite semi-classique}, Comm. Math. Phys \textbf{109} (1987), 313--326. 

\bibitem{Hor1}
L. H\"ormander,
The analysis of linear partial differential operators. I, Grundlehren der Mathematischen
Wissenschaften [Fundamental Principles of Mathematical Sciences], vol. 256, SpringerVerlag,
Berlin, 1983, Distribution theory and Fourier analysis.

\bibitem{Hor2}
L. H\"ormander,
The analysis of linear partial differential operators. I, Grundlehren der Mathematischen
Wissenschaften [Fundamental Principles of Mathematical Sciences], vol. 257, SpringerVerlag,
Berlin, 1983, Differential operators with constant coefficients.

\bibitem{KLY} H. Kang, M. Lim and S. Yu, {\it Spectral resolution of the Neumann-Poincar\'{e} operator on intersecting disks and analysis of plasmon resonance}, Arch. Ration. Mech. Anal., {\bf 226} (2017), 83--115.

\bibitem{seo} 
{H. Kang and J.K. Seo,}
{\em Inverse conductivity problem with one measurement: uniqueness of balls in $\mathbb{R}^3$}, SIAM J. Appl. Math., \textbf{59} (1999), 851--867.

\bibitem{putinar} D. Khavinson, M. Putinar, H.S. Shapiro, {\it Poincar\'e's variational problem in potential theory}, Arch. Ration. Mech. Anal., \textbf{185} (2007), 143--184.

\bibitem{kellog} 
O.D. Kellogg, {\it Foundations of Potential Theory}, Reprint from the first edition of 1929, Die
Grundlehren der Mathematischen Wissenschaften, Band 31 Springer-Verlag, Berlin-New York, 1967.

\bibitem{shapiro} {D. Khavinson, M. Putinar, and H.S. Shapiro}, {\it Poincar\'e's variational problem
in potential theory}, Arch. Ration. Mech. Anal., \textbf{185} (2007), 143--184.

\bibitem{Kli} V. V. Klimov, {\it Nanoplasmonics}, CRC Press, 2014. 





\bibitem{LL1} H. Li and H. Liu, {\it On anomalous localized resonance and plasmonic cloaking beyond the quasi-static limit}, Proc. R. Soc. A, \textbf{474} (2018), 20180165. 

\bibitem{LLL1} H. Li, J. Li and H. Liu, {\it On quasi-static cloaking due to anomalous localized resonance in $\mathbb{R}^3$}, SIAM J.
Appl. Math., \textbf{75} (2015), 1245--1260. 

\bibitem{LZ} B. Luk'yanchuk, N. I. Zheludev, S. A. Maier, N. J. Halas, P. Nordlander, H. Giessen, and C. T. Chong, {\it The Fano resonance in plasmonic nanostructures and metamaterials}, Nature Materials, \textbf{9} (2010), 707.

\bibitem{plasmon1} {I.D. Mayergoyz, D.R. Fredkin, and Z. Zhang}, {\it Electrostatic (plasmon) resonances in nanoparticles}, Phys. Rev. B, \textbf{72} (2005), 155412.

\bibitem{MN} G. W. Milton and N.-A. P. Nicorovici, {\it On the cloaking effects associated with anomalous localized resonance},
Proc. R. Soc. A, \textbf{462} (2006), 3027--3059. 

\bibitem{weyl2}
{Y. Miyanishi,}
{\em Weyl's law for the eigenvalues of the Neumann-Poincar\'e
operators in three dimensions: Willmore energy and surface
geometry,} Preprint, arXiv:1806.03657

\bibitem{weyl1}
{Y. Miyanishi and G. Rozenblum,}
{\em Eigenvalues of the Neumann-Poincar\'e operators in dimension 3: Weyl's Law and geometry, } St. Petersburg Math. J., \textbf{31} (2020), 371--386. 

\bibitem{vonneumann}
J. von Neumann, \textit{Proof of the quasi-ergodic hypothesis}.
Proceedings USA Academy \textbf{18}  (1932), 70- - 82.

\bibitem{osborn} J.E. Osborn, Spectral approximation for compact operators, Math. Comp., 29 (1975), 712--725.


\bibitem{OI} F. Ouyang and M. Isaacson, {\it Surface plasmon excitation of objects with arbitrary shape and dielectric
constant}, Philos. Mag., \textbf{60} (1989), 481--492.

\bibitem{corner}
K.-M. Perfekt
{\it Plasmonic eigenvalue problem for corners: limiting absorption principle and absolute continuity in the essential spectrum},
Journal de Mathematiques Pures et Appliquees (2020), in press.

\bibitem{Duistermaat-Heckman}
E. Prato, S. Wu, {\it Duistermaat-Heckman measures in a noncompact setting}, Compositio Mathematica, \textbf{94} (1994), 113--128. 


\bibitem{Sch} J. A. Schuller, E. S. Barnard, W. Cai, Y. C. Jun, J. S. White and M. L. Brongersma, {\it Plasmonics for
extreme light concentration and manipulation}, Nature Materials, \textbf{9} (2010), 193--204.

\bibitem{erg1} A.I. Shnirelman, \textit{Ergodic properties of eigenfunctions}, Uspehi Mat. Nauk, 29 (1974), 181--182.

\bibitem{erg11} A.I. Shnirelman, \textit{On the asymptotic properties of eigenfunctions in the region of chaotic motion}, addendum to V.F.Lazutkin, KAM theory and semiclassical approximations to eigenfunctions, Springer (1993).

\bibitem{SPW} D. R. Smith, J. B. Pendry and M. C. K. Wiltshire, {\it Metamaterials and negative refractive index}, Science, \textbf{305} (2004),
788--792. 

\bibitem{erg_a} T. Sunada, {\it Quantum ergodicity}, Progress in inverse spectral geometry,  Trends Math., Birkhuser, Basel, 1997, 175--196,

\bibitem{trace} T. Sunada, {\it Trace formula and heat equation asymptotics for a nonpositively curved manifold}, Amer. J. Math. \textbf{104} (1982), 795?812.

\bibitem{uhlmann}
{G. Uhlmann, }
{\em The Dirichlet to Neumann Map and Inverse Problems,} preprint.

\bibitem{ergodic}
P. Walters
An Introduction to Ergodic Theory,
Graduate Texts in Mathematics 79, Springer-Verlag New York, 1982.

\bibitem{erg3} S. Zelditch, {\it Eigenfunctions on compact Riemann-surfaces of $g\geq 2$}, preprint, 1984.

\bibitem{Zel} S. Zelditch, {\it Eigenfunctions of the Laplacian of Riemannian Manifolds}, book in preprint, 2017. 

\bibitem{erg33} S. Zelditch, {\it Quantum ergodicity of $C^∗$ dynamical systems}, Comm. Math. Phys. \textbf{177} (1996), 507--528.

\bibitem{erg34} S. Zelditch, {\it Quantum ergodicity on the sphere}, Comm. Math. Phys. \textbf{146} (1992), 61--71.

\bibitem{erg35} S. Zelditch, {\it  Quantum transition amplitudes for classically ergodic or completely integrable systems},
J. Fun. Anal. \textbf{94} (1990), 415--436.

\bibitem{erg32} S. Zelditch, {\it Uniform distribution of eigenfunctions on compact hyperbolic surfaces}, Duke Math. J. \textbf{55} (1987), 919--941.

\bibitem{Z} S. Zeng, D. Baillargeat, H. P. Ho and K. T. Yong, {\it Nanomaterials enhanced surface plasmon resonance for biological and chemical sensing applications}, Chemical Society Reviews, \textbf{43} (2014), 3426--3452.





\end{thebibliography}
\end{document}